\documentclass[11pt]{article}
\usepackage{amssymb}
\usepackage{amsthm}
\usepackage{amsmath}
\usepackage{float}
\usepackage{appendix}
\usepackage{adjustbox}

\newtheorem{theorem}{Theorem}[section]
\newtheorem{lemma}{Lemma}[section]
\newtheorem{proposition}{Proposition}[section]
\newtheorem{corollary}{Corollary}[section]

\numberwithin{equation}{section}

\newcommand{\cE}{\mathcal{E}}
\newcommand{\ep}{\varepsilon}
\newcommand{\C}{\mathbb{C}}
\newcommand{\Q}{\mathbb{Q}}
\newcommand{\R}{\mathbb{R}}
\newcommand{\F}{\mathbb{F}}

\newcommand{\FF}{\mathbb{F}} 
\newcommand{\Out}{\mathrm{Out}}

\newcommand{\ad}{\mathrm{ad}}
\newcommand{\msc}{\mathrm{sc}}

\newcommand{\bG}{\mathbf{G}}

\newcommand{\bT}{\mathbf{T}}
\newcommand{\bH}{\mathbf{H}}

\newcommand{\Irr}{\mathrm{Irr}}

\def\adots{\mathinner{\mkern2mu\raise0pt\hbox{.}  
\mkern2mu\raise4pt\hbox{.}\mkern1mu
\raise7pt\vbox{\kern7pt\hbox{.}}\mkern1mu}}

\begin{document}

\bibliographystyle{amsplain}

\title{A computational approach to the Frobenius--Schur indicators of finite exceptional groups}
\author{Stephen Trefethen and C. Ryan Vinroot}
\date{}

\maketitle

\begin{abstract}
We prove that the finite exceptional groups $F_4(q)$, $E_7(q)_{\ad}$, and $E_8(q)$ have no irreducible complex characters with Frobenius--Schur indicator $-1$, and we list exactly which irreducible characters of these groups are not real-valued.  We also give a complete list of complex irreducible characters of the Ree groups ${^2 F_4}(q^2)$ which are not real-valued, and we show the only character of this group which has Frobenius--Schur indicator $-1$ is the cuspidal unipotent character $\chi_{21}$ found by M. Geck.
\\
\\
\noindent 2010 {\it Mathematics Subject Classification: } 20C33, 20C40
\end{abstract}

\section{Introduction}

Given a finite group $G$ and its table of complex irreducible characters, it is a natural question to ask the value of the Frobenius--Schur indicator of each irreducible character of $G$.  That is, we may ask whether each character is real-valued, and if it is, whether it is afforded by a representation which is defined over the real numbers.  The Frobenius--Schur indicator is given, for example, in each character table in the {\em Atlas of finite groups} \cite{Atlas}.  If we have the full character table of $G$, and we know the square of each conjugacy class of $G$, then the Frobenius--Schur indicator may be computed directly from its character formula.  Without directly applying the formula, we could also identify which characters of $G$ are not real-valued from its character table, and if we have proven that certain irreducible characters of $G$ have indicator equal to $-1$, we can check if these are the only such characters by counting the involutions in $G$ and using the Frobenius--Schur involution count.  We do exactly this for many finite exceptional groups.  In some cases where the generic character table is not known (for example, $E_8(q)$), we use results from the character theory of finite reductive groups to identify precisely those irreducible characters which are not real-valued, and so we obtain complete results for the Frobenius--Schur indicators for the characters of these groups.  

After preliminary notions and results in Section \ref{Prelim}, we give examples of this method for some small-rank exceptional groups in Section \ref{ExamplesSection}, namely for the groups ${^3 D_4(q)}$, ${^2 B_2(q^2)}$, ${^2 G_2(q^2)}$, and $G_2(q)$.  We confirm some known results on the Frobenius--Schur indicators in these groups, which are all $1$ or $0$ in these cases.  We also discuss two related problems for these groups in Section \ref{ExamplesSection}.  First is the computation of the Schur index of the irreducible characters of a group; the index is conjectured to be at most $2$ for any finite quasisimple group (see Section \ref{SchurSec}), and is known to be $1$ for every irreducible character of the exceptional groups just listed.  The second related problem is the determination of the strongly real classes of a group, which are the real classes which can be inverted by an involution.  It is known that every real class is strongly real for each of the above finite exceptional groups, except for $G_2(q)$ when $q$ is a power of $2$ or $3$.  In Proposition \ref{G223}, we confirm this statement for those remaining cases by using the known generic character table of $G_2(q)$.

In Section \ref{F4Section}, we study the exceptional groups $F_4(q)$ and ${^2 F_4(q^2)}$.  While the full generic character table of $F_4(q)$ is not yet fully known, the fields of character values and the Schur indices of unipotent characters of $F_4(q)$ (and all other exceptional groups) have been determined by Geck \cite{Ge03, Ge04, Ge05}.  It turns out that the only characters of $F_4(q)$ which are not real-valued are unipotent, and we show that the other characters all have Frobenius--Schur indicator 1 in Theorem \ref{F4}.  As a corollary, we check that the Schur index of every character of $F_4(q)$ is at most $2$ in Corollary \ref{F4Schur}.  The generic character table of ${^2 F_4(q^2)}$ can be accessed in the package {\sf CHEVIE} \cite{CHEVIE}, and much of its character table is available in the literature \cite{HiHu12, Ma90}. With this information, and the work of Geck on unipotent characters, we compute the Frobenius--Schur indicators for the characters of ${^2 F_4(q^2)}$ in Theorem \ref{2F4}.  In particular we show that the only character of this group which has indicator equal to $-1$ is the cuspidal unipotent character found by Geck \cite{Ge03}.

We develop some results in Section \ref{AutSection} on the character theory of finite reductive groups which are applied to extend our method to some exceptional groups of larger rank in Section \ref{E7E8Section}.  Lemma \ref{GalAct} gives some conditions on arbitrary connected reductive groups over finite fields to ensure all unipotent characters are invariant under any rational outer automorphism.  We combine this with a previous result on the Jordan decomposition of real-valued characters of certain finite reductive groups \cite{SV15} to give manageable conditions for an irreducible character to be real-valued in Lemma \ref{RealLemma}.  These results, along with the work of Geck on unipotent characters and computations of L\"{u}beck \cite{LuWWW1, LuWWW} allow us to compute the Frobenius--Schur indicators of all irreducible characters of $E_7(q)_{\ad}$ and $E_8(q)$ in Theorems \ref{E7} and \ref{E8}.  In particular, all real-valued characters of these groups have indicator $1$, and a list of all characters of these groups which are not real-valued is given, along with their character degrees, in Tables \ref{E7Table} and \ref{E8Table}.

In Section \ref{RemarksSec}, we discuss the remaining cases and what results and information are needed in order to complete the computations of Frobenius--Schur indicators for these finite exceptional groups.  Finally, all tables of the results from the computations made are given in the Appendix.\\
\\
\noindent{\bf Acknowledgements. }  The authors thank Frank Himstedt for helpful communication regarding accessing the generic character table of ${^2 F_4(q^2)}$ in {\sf CHEVIE}, Mandi Schaeffer Fry for pointing out the exceptional cases needed in Lemma \ref{GalAct}, and Eamonn O'Brien for various corrections.  The authors thank the anonymous referee for helpful comments and corrections, including some corrections to Table \ref{G2CMC} in the case $q=2$.  Vinroot was supported in part by a grant from the Simons Foundation, Award \#280496.

\section{Preliminaries} \label{Prelim}

\subsection{Frobenius--Schur indicators and Schur indices} \label{SchurSec}

If $G$ is a finite group, we denote by $\Irr(G)$ the set of irreducible complex characters of $G$.  Given $\chi \in \Irr(G)$, we recall the {\em Frobenius--Schur indicator} of $\chi$ (see \cite[Chapter 4]{IsBook}), which we denote by $\ep(\chi)$, may be defined as
$$\ep(\chi) = \frac{1}{|G|} \sum_{g \in G} \chi(g^2).$$
Then $\ep(\chi) = 1, -1,$ or $0$, where $\ep(\chi) = 0$ precisely when $\chi$ is not real-valued.  When $\chi$ is real-valued, $\ep(\chi) = 1$ precisely when $\chi$ is afforded by a complex representation which may be defined over the real numbers, and otherwise $\ep(\chi) = -1$.  The {\em Frobenius--Schur involution count} is the equation \cite[Corollary 4.6]{IsBook}
\begin{equation} \label{FScount}
 \sum_{\chi \in \Irr(G)} \ep(\chi) \chi(1) = \# \{ g \in G \, \mid \, g^2 = 1 \}.
\end{equation}
Every $g \in G$ such that $g^2 = 1$ is an {\em involution} of $G$, where we include the identity as an involution for convenience.  Note that from \eqref{FScount}, since $\ep(\chi) \leq 1$ for each $\chi \in \Irr(G)$, the sum of the degrees of the real-valued characters $\chi \in \Irr(G)$ is always at least the number of involutions, and there is equality precisely when $\ep(\chi) = 1$ for all real-valued $\chi \in \Irr(G)$.  This observation is the main strategy in computing Frobenius--Schur indicators in this paper.

If $K$ is a subfield of $\C$, and $\chi \in \Irr(G)$, then let $K(\chi)$ denote the smallest field containing $K$ and all of the values of $\chi$.  The {\em Schur index} of $\chi$ over $K$, which we denote by $m_{K}(\chi)$, is the smallest positive integer $m$ such that $m\chi$ is afforded by a representation defined over $K(\chi)$ (see \cite[Corollary 10.2]{IsBook}).  In particular, if $K = \R$ then $m_{\R}(\chi) = 1$ when $\ep(\chi)=1$ or $\ep(\chi) = 0$, and $m_{\R}(\chi) = 2$ when $\ep(\chi) = -1$.  For the case $K= \Q$, it has been conjectured that if $G$ is a quasisimple group, then $m_{\Q}(\chi) \leq 2$ for all $\chi \in \Irr(G)$, while it is known that in general $m_{\Q}(\chi)$ might be arbitrarily large for other finite groups, see \cite{Tu02} for example.  A relevant result relating this conjecture to our results in this paper is the Brauer--Speiser Theorem (see \cite[p.~171]{IsBook}), which states that if $\chi \in \Irr(G)$ and $\chi$ is real-valued, then $m_{\Q}(\chi) \leq 2$. 

\subsection{Exceptional groups} 

Here we recall basic facts about finite exceptional groups and finite simple groups of exceptional type.  Let $p$ be a prime, $\F_p$ a finite field with $p$ elements and $\overline{\F}_p$ a fixed algebraic closure.  Let $\bG$ be a connected reductive group over $\overline{\F}_p$, and let $F$ be a Steinberg map of $\bG$.  We consider the {\em exceptional groups} to be the groups of the form $\bG^F$ with $\bG$ a simple algebraic group, and either the isogeny type of $\bG$ is one of $G_2$, $F_4$, $E_6$, $E_7$, $E_8$, or  $\bG^F$ is either a Steinberg triality group ${^3 D_4}(q)$ with $q$ an integer power of $p$, or a Suzuki group of the form ${^2 B_2}(q^2)$ where $q^2 = 2^{2m+1}$ and $p=2$.  In the former case, $\bG^F$ can be one of the groups (see \cite[Section 1.19]{Ca85})
$$G_2(q), F_4(q), E_6(q)_{\ad}, E_6(q)_{\msc}, {^2 E_6}(q)_{\ad}, {^2 E_6}(q)_{\msc}, E_7(q)_{\ad}, E_7(q)_{\msc}, E_8(q),$$
with $q$ an integer power of $p$, or a Ree group of the form ${^2 G_2(q^2)}$ with $q^2 = 3^{2m+1}$ and $p=3$, or ${^2 F_4}(q^2)$ with $q^2 = 2^{2m+1}$ and $p=2$.

We recall that the groups $G_2(q)$ (unless $q=2$), $F_4(q)$, and $E_8(q)$ are finite simple groups, with the associated algebraic group $\bG$ having trivial (and so connected) center over $\overline{\F}_p$. The derived subgroup, $G_2(2)^{\prime}$, of $G_2(2)$ is simple, and is isomorphic to $\mathrm{PSU}(3,3)$. The groups ${^3 D_4}(q)$, ${^2 B_2}(q^2)$, and ${^2 G_2}(q^2)$ are finite simple groups, and ${^2 F_4}(q^2)$ is a finite simple group unless $q^2 = 2$, in which case the derived subgroup ${^2 F_4}(2)^{\prime}$ of index 2, or the Tits group, is a finite simple group.  The simple algebraic groups of adjoint type have trivial (and so connected) center, but the groups of simply connected type do not in general.  If $q \not\equiv 1($mod $3)$, then $E_6(q)_{\ad} = E_6(q)_{\msc} = E_6(q)$ is a finite simple group, and otherwise $E_6(q)_{\ad}$ has the finite simple group $E_6(q)$ as an index $3$ subgroup, and $E_6(q)_{\msc}$ has an order $3$ center $Z$ and $E_6(q)_{\msc}/Z$ is isomorphic to the finite simple group $E_6(q)$.  If $q \not\equiv -1($mod $3)$, then ${^2 E_6}(q)_{\ad} = {^2 E_6}(q)_{\msc} = {^2 E_6}(q)$ is a finite simple group, and otherwise ${^2 E_6}(q)_{\ad}$ has the finite simple group ${^2 E_6}(q)$ has an index 3 subgroup, and ${^2 E_6}(q)_{\msc}$ has an order $3$ center $Z$ with quotient ${^2 E_6}(q)_{\msc}/Z$ isomorphic to the finite simple group ${^2 E_6}(q)$.  Finally, if $q$ is even then $E_7(q)_{\ad} = E_7(q)_{\msc} = E_7(q)$ is a finite simple group, while if $q$ is odd then $E_7(q)_{\ad}$ has as an index 2 subgroup the finite simple group $E_7(q)$, and $E_7(q)_{\msc}$ has an order $2$ center $Z$ with quotient $E_7(q)_{\msc}/Z$ isomorphic to the finite simple group $E_7(q)$.

\subsection{Jordan decomposition of characters} \label{JDSection}

Let $\bG$ be a connected reductive group over $\bar{\F}_p$ with Steinberg map $F$.  If we take $\bT$ to be a maximal $F$-stable torus of $\bG$, with $\theta$ an irreducible character of $\bT^F$, then we denote by $R_{\bT}^{\bG}(\theta)$ the Deligne--Lusztig generalized character of $\bG^F$ defined by the $\bG^F$-conjugacy class of pairs $(\bT, \theta)$ (see \cite[Chapter 7]{Ca85} or \cite[Chapter 11]{dmbook}).

The \emph{unipotent characters} of $\bG^F$ are exactly the irreducible characters $\psi$ of $\bG^F$ such that $\langle \psi, R_{\bT}^{\bG}(\mathbf{1}) \rangle \neq 0$ for some $F$-stable maximal torus $\bT$ of $G$,  where $\mathbf{1}$ is the trivial character and $\langle \cdot, \cdot \rangle$ denotes the standard inner product on class functions of a finite group.  If $\bG$ is a disconnected group with connected component $\bG^{\circ}$, then the unipotent characters of $\bG^F$ are taken to be the irreducible constituents of $\mathrm{Ind}_{(\bG^{\circ})^F}^{\bG^F} (\psi)$, where $\psi$ varies over all unipotent characters of $(\bG^{\circ})^F$.

Now let $\bG^*$ be some fixed dual group of $\bG$ (with respect to some fixed Lie root data for $\bG$), with corresponding dual Steinberg map $F^*$.  By \cite[Proposition 13.13]{dmbook}, there is a bijection between $\bG^{*F^*}$-conjugacy classes of pairs $(\bT^*, s)$ where $\bT^*$ is an $F^*$-stable maximal torus of $\bG^*$, and $s \in \bT^{*F^*}$ is a semisimple element of $\bG^{*F^*}$, and $\bG^F$-conjugacy classes of pairs $(\bT, \theta)$ where $\theta$ is an irreducible character of $\bT^F$.  Given a semisimple $\bG^{*F^*}$-conjugacy class $(s)$, the {\em rational Lusztig series} $\cE(\bG^F, (s))$ of $\bG^F$ is the collection of irreducible characters $\chi \in \Irr(\bG^F)$ such that $\langle \chi, R_{\bT^F}^{\bG^F}(\theta) \rangle \neq 0$ where the $\bG^F$-class of the pair $(\bT, \theta)$ corresponds to the $\bG^{*F^*}$-class of the pair $(\bT^*, s)$ for some $F^*$-stable maximal torus $\bT^*$ of $\bG^{*F^*}$.  In particular, the collection of unipotent characters of $\bG^F$ is given by the Lusztig series $\cE(\bG^F, (1))$, since the classes of pairs $(\bT, \mathbf{1})$ and $(\bT^*, 1)$ correspond when $\bT$ and $\bT^*$ are dual tori.

For $\bG$ a connected reductive group, given a semisimple class $(s)$ of $\bG^{*F^*}$ and some fixed element $s$ from that class, a {\em Jordan decomposition map} is a bijection
$$J_s: \cE(\bG^F, (s)) \longrightarrow \cE(C_{\bG^*}(s)^{F^*}, (1)) \;\; \text{ such that } \;\; \langle \chi, R_{\bT}^{\bG}(\theta) \rangle = \pm \langle J_s (\chi), R_{\bT^*}^{C_{\bG^*}(s)}({\bf 1}) \rangle$$
for every $\chi \in \cE(\bG^F, (s))$, where $C_{\bG^*}(s)$ is the centralizer, and $\bT^*$ is a maximal $F^*$-stable torus such that $s \in \bT^{*F^*}$.  Such a bijection was shown to always exist when $Z(\bG)$ is connected by Lusztig \cite{Lu84}.  While we do not explicitly need it in this paper, such a map also always exists when $Z(\bG)$ is not connected, where if $C_{\bG^*}(s)$ is disconnected we still let $\cE(C_{\bG^*}(s)^{F^*}, (1))$ denote its set of unipotent characters (see \cite[Theorem 13.23]{dmbook}).  Using the maps $J_s$, we may parameterize the irreducible characters of $\bG^F$ by $\bG^{*F^*}$-conjugacy classes of pairs $(s, \psi)$ where $\psi$ is a unipotent character of $C_{\bG^*}(s)^{F^*}$, where the class of the pair $(s, \psi)$ corresponding to $\chi \in \Irr(\bG^F)$ is called its {\em Jordan decomposition}.  The most important property of the Jordan decomposition of characters that we will need is that if $\chi$ corresponds to the class of pairs $(s, \psi)$, then the degree of the character $\chi$ is given by $\chi(1) = [\bG^{*F^*}: C_{\bG^*}(s)^{F^*}]_{p'} \psi(1)$,
where the subscript $p'$ denotes the prime-to-$p$ part of that centralizer \cite[Remark 13.24]{dmbook}.

\section{Examples and related results} \label{ExamplesSection}

In this section, we consider the Frobenius--Schur indicators of the characters of the groups ${^3 D_4}(q)$, ${^2 B_2}(q^2)$, ${^2 G}_2(q^2)$, and $G_2(q)$.  While these results are known, we give these examples as motivation for the computational method used for other results in this paper, and to place the results in the context of other interesting questions.\\
\\
\noindent \emph{The group ${^3 D_4(q)}$}.  When $q$ is odd, Barry \cite[Step 1]{Ba88} makes exactly the computation which we carry out for other examples.  Namely, it is shown that when $q$ is odd, the group has $q^{16}+q^{12}+q^8+1$ involutions, and that this matches the sum of the character degrees which can be obtained from \cite[Table 4.4]{DeMi87}.  When $q$ is even, it follows from \cite[Section 8]{Th70} or \cite[Section 18]{AsSe76} that the number of involutions in ${^3 D_4(q)}$ is $q^{16} + q^{12}-q^4$.  It follows from \cite[Table 4.4]{DeMi87} that this is also the sum of the character degrees of the group when $q$ is even.  Thus $\ep(\chi) = 1$ for all irreducible characters of ${^3 D_4(q)}$ by the Frobenius--Schur involution count \eqref{FScount}.\\
\\
\noindent \emph{The group ${^2 B_2(q^2)}$, $q^2 = 2^{2m+1}$}.  The classes and characters of this group are computed by Suzuki \cite{Su60, Su62}.  It follows from \cite[Proposition 7]{Su62} that the number of involutions in ${^2 B_2(q^2)}$ is $q^6-q^4+q^2$.  From \cite[Theorem 13]{Su62}, the two characters labeled as $W_l$ are not real-valued.  The sum of the degrees of the remaining characters is $q^6 - q^4 + q^2$, and so $\ep(\chi)=1$ for all real-valued irreducible characters of ${^2 B_2(q^2)}$ again by \eqref{FScount}.\\
\\
\noindent \emph{The group ${^2 G_2(q^2)}$, $q^2 = 3^{2m+1}$}.  It follows from \cite[Theorem 8.5]{Re61} and \cite[Introduction]{Wa66} that the number of involutions in ${^2 G_2(q^2)}$ is $q^8 - q^6+q^4+1$.  The characters which are not real-valued are given in the table of Ward \cite[p.~87]{Wa66}, labeled as $\xi_5$, $\xi_6$, $\xi_7$, $\xi_8$, $\xi_9$, and $\xi_{10}$.  Their degrees are given, as are all of the degrees for the irreducible characters of ${^2 G_2(q^2)}$.  Taking the sum of the degrees of the real-valued irreducible characters yields the number of involutions in the group, and so $\ep(\chi) = 1$ for all real-valued $\chi \in \Irr({^2 G_2(q^2)})$.\\
\\
\noindent \emph{The group $G_2(q)$}.  First, when $q$ is odd it follows from \cite[Theorem 4.4 and p.~209]{Ch68} and \cite[Table 2]{En69} or from \cite[p.~282]{Iw70} that the number of involutions in $G_2(q)$ is $q^8+q^6+q^4+1$.  It follows from \cite[Propositions 2.5 and 2.6]{En69} that when $q$ is even the number of involutions in $G_2(q)$ is $q^8 + q^6 -q^2$.  Next, it follows from \cite[Table 1]{Ge03} that the cuspidal unipotent characters $G_2[\theta]$ and $G_2[\theta^2]$ are not real-valued, and from \cite[p.~478]{Ca85} we find these characters each have degree $\frac{1}{3} q(q^2 - 1)^2$.  We may take the sum of the degrees of all other characters, by using the tables in \cite{ChRe74, En76} for $q$ odd and \cite{EnYa86} for $q$ even (or by using \cite{LuWWW}), and we find that this matches the number of involutions.  Thus $\ep(\chi) = 1$ for each real-valued irreducible character $\chi$ of $G_2(q)$. \\
\\
\indent We mention other relevant results for the groups in the above examples.  It is known that for every $G$ above, every $\chi \in \Irr(G)$ satisfies $m_{\Q}(\chi) = 1$.  This is proved for $G = {^2 B_2(q^2)}$ or ${^2 G_2(q^2)}$ by Gow \cite[Theorem 9]{Go762}, for $G = {^3 D_4(q)}$ by Barry when $q$ is odd \cite{Ba88} and by Ohmori when $q$ is even \cite[Theorem 3]{Oh03}, and for $G= G_2(q)$ when $q$ is odd by Ohmori \cite{Oh85} and when $q$ is even by Enomoto and Ohmori \cite{EnOh93}.  In fact, these results on the Schur index imply that the Frobenius--Schur indicator is $1$ for each irreducible real-valued character of each of these groups. 

Recall that an element $g$ of a finite group $G$ is {\em real} if $g$ is conjugate to $g^{-1}$ in $G$, and that the number of conjugacy classes of real elements in $G$ is equal to the number of real-valued characters in $\Irr(G)$.  We may also ask whether every conjugacy class of real elements is necessarily {\em strongly real} for each of the groups $G$ above, where an element (or conjugacy class) $g$ of a group $G$ is strongly real in $G$ if there exists $h \in G$ such that $h^2 = 1$ and $h^{-1} g h = g^{-1}$.  For $G = {^3 D_4(q)}$, it is proved by Vdovin and Gal$^\prime$t \cite[Theorem 1]{VdGa10} that all classes of $G$ are strongly real.  Suzuki proved \cite[Section 10]{Su62} that all real classes of the groups ${^2 B_2(q^2)}$ are strongly real.  For $G = {^2 G_2(q^2)}$, we invoke a result of Gow \cite[Corollary 1]{Go761} which says that if a finite group has an abelian Sylow $2$-subgroup, then all of its real elements are strongly real if and only if all real-valued irreducible characters of the group have Frobenius--Schur indicator $1$.  The latter holds for these groups as discussed above, and the Sylow 2-subgroup is abelian by \cite[Theorem 8.5]{Re61}, and so all real classes are strongly real in ${^2 G_2(q^2)}$.

Singh and Thakur \cite[Corollary A.1.6]{SiTh08} proved that if $q$ is not a power of $2$ or $3$, then all real elements of $G_2(q)$ are strongly real.  We now address the remaining cases.

\begin{proposition} \label{G223} All real classes of $G_2(q)$ are strongly real. 
\end{proposition}
\begin{proof} As just mentioned, Singh and Thakur prove this statement when $q$ is not a power of $2$ or $3$.  The generic character table for $G_2(q)$ when $q$ is a power of $2$ or $3$ is in the package {\sf CHEVIE} \cite{CHEVIE}.  We may use the character table to prove the statement by using the following result (see \cite[p.~125]{LuPa10}).  If $C_1$, $C_2$, and $C_3$ are conjugacy classes of the finite group $G$, then the number $n_{123}$ of pairs $(g_1,g_2)$ such that $g_1 \in C_1$, $g_2 \in C_2$, and the product $g_1 g_2$ is equal to a fixed element $g_3 \in C_3$ is given by
$$n_{123} = \frac{|C_1||C_2|}{|G|} \sum_{ \chi \in \Irr(G)} \frac{\chi(C_1) \chi(C_2) \overline{\chi(C_3)}}{\chi(1)}.$$
These are the \emph{class multiplication coefficients}, and we must show that for every real class $C_3$ of $G_2(q)$, there are classes $C_1$ and $C_2$ of involutions such that $n_{123} \neq 0$.  This is indeed the case for $G_2(q)$ with $q$ a power of $2$ or $3$, and our results from the computation are given in Table \ref{G2CMC}.
\end{proof}

\section{The groups $F_4(q)$ and ${^2 F_4(q^2)}$} \label{F4Section}

We now consider the finite exceptional group $F_4(q)$.  We have the following result, where our notation for the cuspidal unipotent characters of $F_4(q)$ is that of Lusztig \cite{Lu84}.

\begin{theorem} \label{F4}  For every prime power $q$, the only irreducible characters of $F_4(q)$ which are not real-valued are the cuspidal unipotent characters $F_4[i]$, $F_4[-i]$, $F_4[\theta]$, and  $F_4[\theta^2]$.  All other irreducible characters $\chi$ of $F_4(q)$ satisfy $\ep(\chi) = 1$.  That is, $m_{\R}(\chi) = 1$ for all $\chi \in \Irr(F_4(q))$.
\end{theorem}
\begin{proof} When $q$ is odd, the number of involutions in $F_4(q)$ can be computed using \cite{Iw70} or \cite{Sho74}, and is given by
$$ q^{28} + q^{26} + 2q^{24} + 2q^{22} + 2 q^{20} + 2q^{18} + 2q^{16} +q^{14} + q^{12} + q^8.$$
When $q$ is even, the number of involutions in $F_4(q)$ may be computed using \cite[Corollary 1]{Sh74} or \cite[Section 13]{AsSe76}, and is given by
$$ q^{28} + q^{26} + q^{24} + q^{22} -q^{20} -q^{16} - q^{14} - q^{10} + q^8.$$
The fact that the four listed cuspidal unipotent characters of $F_4(q)$ are not real-valued (independently of $q$) follows from Geck \cite[Table 1]{Ge03}, and the degrees of these characters can be found in the table of Lusztig \cite[p.~372]{Lu84}.  Taking the sum of all of the character degrees of $F_4(q)$ using the data in \cite{LuWWW}, with the result listed in Table \ref{SumTable}, and subtracting the degrees of the unipotent characters which are not real-valued, we obtain precisely the numbers of involutions given above, whether $q$ is even or odd.  The result follows.
\end{proof}

As mentioned in Section \ref{SchurSec}, it has been conjectured that if $G$ is a quasisimple group, then $m_{\Q}(\chi) \leq 2$ for all $\chi \in \Irr(G)$.  As an application of Theorem \ref{F4} and previous work of Geck \cite{Ge03, Ge04}, we are able to conclude this statement indeed holds for the case that $G = F_4(q)$.

\begin{corollary} \label{F4Schur} For every prime power $q$, and every $\chi \in \Irr(F_4(q))$, we have $m_{\Q}(\chi) \leq 2$.
\end{corollary}
\begin{proof} It is proved by Geck in \cite[Table 1 and Section 6]{Ge03} and \cite[Corollary 3.2]{Ge04} that if $\chi$ is one of the cuspidal unipotent characters $F_4[i]$, $F_4[-i]$, $F_4[\theta]$, or $F_4[\theta^2]$, then $m_{\Q}(\chi) = 1$.  By Theorem \ref{F4}, the rest of the irreducible characters of $F_4(q)$ are real-valued, and so the result now follows from the Brauer--Speiser Theorem. \end{proof}

In fact, it follows from the work of Geck \cite{Ge03, Ge04} that $m_{\Q}(\chi)=1$ for all unipotent characters $\chi$ of $F_4(q)$.  We expect this stronger result to hold for all irreducible characters of $F_4(q)$.  We also expect that all real classes of $F_4(q)$ are strongly real, which can be checked for the case $F_4(2)$ using {\sf GAP} \cite{GAP} as in the proof of Proposition \ref{G223}.  

We now consider the Ree groups ${^2 F_4}(q^2)$ with $q^2 = 2^{2m+1}$.  In the following, the notation for the unipotent characters of ${^2 F_4(q^2)}$ is taken from the paper of Malle \cite{Ma90}, and the notation for the non-unipotent characters is that used in the paper of Himstedt and Huang \cite{HiHu12}.

\begin{theorem} \label{2F4} The only irreducible character of ${^2 F_4(q^2)}$ with Frobenius--Schur indicator $-1$ is the cuspidal unipotent character $\chi_{21}$.  The only irreducible characters which are not real-valued are the unipotent characters $\chi_5$, $\chi_6$, $\chi_7$, $\chi_8$, $\chi_{15}$, $\chi_{16}$, $\chi_{17}$, and $\chi_{18}$, and the non-unipotent characters ${_G\chi_{43}(k)}$, ${_G\chi_{44}(k)}$, ${_G\chi_{47}(k)}$, and ${_G\chi_{48}(k)}$.
\end{theorem}
\begin{proof} First, the classes of involutions in the group are given in \cite[Corollary 2]{Sh74}, and the orders of their centralizers can be obtained from \cite[Theorem 2.1]{Sh74}.  Taking the sum of the indices of these centralizers and adding $1$ yields that the total number of involutions in ${^2 F_4(q^2)}$ is given by
$$ q^{28} - q^{26}+q^{24}+q^{22}-q^{20}+q^{16}-q^{14}+q^{10}-q^8.$$

The fact that the cuspidal unipotent character $\chi_{21}$ satisfies $\ep(\chi_{21})=-1$ is a result of Geck \cite[Theorem 1.6]{Ge03}, also given by Ohmori \cite{Oh03}.  That the listed unipotent characters are not real-valued follows from the work of Malle, where these unipotent characters take non-real values on the unipotent classes listed as $u_3$ in \cite[Tabelle 2]{Ma90}.  The non-unipotent characters listed take non-real values, as computed in the paper of Himstedt and Huang \cite[Table B.12]{HiHu12}, on the class listed as $c_{1, 11}$.  The degrees of these non-unipotent characters are listed in \cite[Table A.14]{HiHu12}.  

The sum of all character degrees of ${^2 F_4}(q^2)$ may be computed using the data in \cite{LuWWW}, and the result is listed in Table \ref{SumTable}.  From this, we subtract the degrees of the characters which are not real-valued, and we subtract twice the degree of the character $\chi_{21}$ with Frobenius--Schur indicator $-1$.  The result is precisely the number of involutions in the group, and so the claim follows by the Frobenius--Schur involution count.
\end{proof}

We may also consider properties of the conjugacy classes of ${^2 F_4}(q^2)$ by using its generic character table in {\sf CHEVIE}.  This group has two classes of involutions, which are {\em class type 2} and {\em class type 3} in {\sf CHEVIE}.   By computing the class multiplication coefficients as in the proof of Proposition \ref{G223}, we find that each of the conjugacy classes labeled as {\em class types 7}, {\em 10}, {\em 11}, and {\em 15} is real but not strongly real.  We thank Frank Himstedt for assistance in this calculation.

For the Tits group ${^2 F_4(2)'}$, it can be checked using {\sf GAP} \cite{GAP} that there are 16 real-valued characters and 6 characters that are not real-valued. All of the real-valued characters have Frobenius--Schur indicator 1, and each of its real conjugacy classes is strongly real.

\section{Automorphisms and unipotent characters} \label{AutSection}

Let $\bG$ be a connected reductive group over $\overline{\FF}_p$ and let $F$ be a Steinberg map for $\bG$.  An automorphism $\sigma$ of the algebraic group $\bG$ which commutes with $F$ is said to be \emph{defined over $F$}.  Then $\sigma$ defines an automorphism of the finite group $\bG^F$.  We let $\Out(\bG, F)$ denote the collection of all outer automorphisms of $\bG$ which are defined over $F$.

\begin{lemma} \label{GalAct} Let $\bG$ be a connected reductive group over $\overline{\FF}_p$ such that no two simple factors of $\bG$ are isogenous, and such that $\bG$ has no simple factor which is of type $D_{2n}$.  Further, if $p=2$ assume $\bG$ has no simple factors of type $C_2$ or $F_4$, and if $p=3$ assume $\bG$ has no simple factors of type $G_2$.  Let $F$ be a Steinberg map for $\bG$, and $\sigma \in \Out(\bG, F)$.  If $\psi$ is a unipotent character of $\bG^F$, then ${^\sigma \psi} = \psi$.
\end{lemma}
\begin{proof} First consider the case that $\bG$ is a simple algebraic group with Steinberg map $F$, and $\bG$ is not of type $D_{2n}$, not of type $C_2$ or $F_4$ if $p=2$, and not of type $G_2$ if $p=3$.  It follows from \cite[Proposition 3.7]{Ma07},  \cite[Theorem 2.5]{Ma08}, and  \cite[Lemma 1.64]{JT12}, that ${^\sigma \psi} = \psi$ for every unipotent character of $\bG^F$ and every $\sigma \in \Out(\bG, F)$.  If $\bG$ is simple of adjoint type, suppose now that $\sigma$ is defined over $F$ but inner, say defined by $x \in \bG$.  Then for every $g \in \bG$, we have $x (F(g)) x^{-1} = F(x) F(g) F(x)^{-1}$, so that $x^{-1} F(x) \in Z(\bG)$.  Since $\bG$ is of adjoint type, $Z(\bG) = 1$, and it follows that $x \in \bG^F$.  Thus ${^\sigma \psi} = \psi$ for every unipotent character $\psi$ of $\bG^F$ when $\sigma$ is an automorphism of $\bG$ defined over $F$ and $\bG$ is a simple algebraic group of adjoint type.  

Next assume that $\bG$ is a connected reductive group of adjoint type, so that $\bG$ is a direct product of simple algebraic groups of adjoint type, say $\bG = \prod_i \bH_i$, and suppose $\sigma$ is an automorphism of $\bG$ which is defined over $F$.  We also assume that no two simple factors of $\bG$ are isogenous, no simple factor $\bH_i$ is of type $D_{2n}$, none of type $C_2$ or $F_4$ if $p=2$, and none of type $G_2$ if $p=3$.  Since $\sigma$ must map simple factors of $\bG$ to other simple factors, and since no pair of simple factors is isogenous, each $\bH_i$ must be $\sigma$-stable.  Similarly, the assumption that no pair of simple factors of $\bG$ is isogenous implies that each simple factor $\bH_i$ is $F$-stable.  Now $\bG^F = \prod_i \bH_i^F$, and it follows that each $\bH_i^F$ is $\sigma$-stable, and so we may view $\sigma$ restricted to $\bH_i$ as an automorphism of $\bH_i$ defined over $F$.  Each unipotent character $\psi$ of $\bG^F$ is of the form $\prod_i \psi_i$ with $\psi_i$ a unipotent character of $\bH_i^F$ (by \cite[p.\ 28]{Lu78}, for example), and by the simple algebraic group case of adjoint type we have ${^\sigma \psi_i} = \psi_i$.  Since ${^\sigma \psi} = \prod_i {^\sigma \psi_i}$, we have ${^\sigma \psi} = \psi$.

Finally we consider the case that $\bG$ is a connected reductive group with no simple factor of type $D_{2n}$, none of type $C_2$ or $F_4$ if $p=2$, none of type $G_2$ if $p=3$, no isogenous pair of simple factors, and $\sigma \in \Out(\bG, F)$.  Consider the adjoint quotient map, which is an algebraic surjection $\phi: \bG \rightarrow \bG_{\ad}$ with $\mathrm{ker}(\phi) = Z(\bG)$.  See \cite[Section 1.5]{GeMa18} for the definition and properties of the adjoint quotient map.  Here $\bG_{\ad}$ is a group of adjoint type, and so is a direct product of simple algebraic groups of adjoint type, where these simple factors are the adjoint types of the simple factors of $\bG$, and so with corresponding factors isogenous.  Thus $\bG_{\ad}$ has no pair of simple factors which are isogenous, and no simple factor of type $D_{2n}$, or of type $C_2$ or $F_4$ if $p=2$, or of type $G_2$ if $p=3$, since this holds for $\bG$.  The adjoint quotient map $\phi$ also induces a Steinberg map on $\bG_{\ad}$, which we also call $F$, which commutes with $\phi$.  We may then define an automorphism $\tilde{\sigma}$ of $\bG_{\ad}$, where $\tilde{\sigma}(g_0) = \phi(\sigma(g))$, where $g \in \bG$ satisfies $\phi(g) = g_0$.  Note that this makes $\tilde{\sigma}$ well-defined since $Z(\bG) = \mathrm{ker}(\phi)$, and $\sigma(Z(\bG)) = Z(\bG)$.  It follows that $F$ commutes with $\tilde{\sigma}$ on $\bG_{\ad}$, that is, $\tilde{\sigma}$ is defined over $F$ on $\bG_{\ad}$.  From the previous case, if $\psi_0$ is a unipotent character of $\bG_{\ad}^F$, then $\,^{\tilde{\sigma}} \psi_0 = \psi_0$.

Now consider a unipotent character $\psi$ of $\bG^F$.  The adjoint quotient map $\phi$ induces a map from $\bG^F$ to $\bG_{\ad}^F$ which has image $G_1$ isomorphic to $\bG^F/Z(\bG^F)$.  By a result of Lusztig \cite[Proposition 3.15]{Lu78}, every unipotent character of $\bG^F$ is obtained by restricting a unipotent of $\bG_{\ad}^F$ to $G_1$ and factoring through the surjection from $\bG^F$.  That is, given unipotent $\psi$ of $\bG^F$, and $g \in \bG^F$, there exists a unipotent character $\psi_0$ of $\bG_{\ad}^F$ and an element $g_1$ of $G_1$ such that $\phi(g) = g_1$ and $\psi(g) = \psi_0(g_1)$.  Then $\,^\sigma \psi(g) = \psi(\sigma(g)) = \psi_0(\tilde{\sigma}(g_1)) = \,^{\tilde{\sigma}} \psi_0(g_1) = \psi_0(g_1)$, since $\,^{\tilde{\sigma}} \psi_0 = \psi_0$.  Since $\psi_0(g_1) = \psi(g)$, we now have $\, ^\sigma \psi = \psi$ as claimed.
\end{proof}

Our main application of Lemma \ref{GalAct} will be in combination with the following result from \cite[Theorem 4.1]{SV15}.

\begin{theorem} \label{RealJord} Suppose that $\bG$ is a connected reductive group with connected center and Frobenius map $F$, and $\chi \in \Irr(\bG^F)$ with Jordan decomposition $(s, \psi)$.  Then $\bar{\chi}$ has Jordan decomposition $(s^{-1}, \bar{\psi})$.  In particular, $\chi$ is real-valued if and only if $s$ is conjugate to $s^{-1}$ in $\bG^{*F^*}$, and if $h \in \bG^{*F^*}$ is such that $hsh^{-1} = s^{-1}$, then ${^h \psi} = \bar{\psi}$.
\end{theorem}

We may now make the following observation, which is crucial in the proof of our main results in the next section.

\begin{lemma} \label{RealLemma} Suppose $\bG$ is a connected reductive group with connected center and Frobenius map $F$, and $\chi \in \Irr(\bG^F)$ with Jordan decomposition $(s, \psi)$.  Suppose that $s$ is a real semisimple element of $\bG^{*F^*}$ with centralizer $C_{\bG^*}(s)$ which has no pair of simple factors which are isogenous, no simple factor of type $D_{2n}$, none of type $C_2$ or $F_4$ if $p=2$, and none of type $G_2$ if $p=3$.  Then $\chi$ is real-valued if and only if $\psi$ is real-valued.
\end{lemma}
\begin{proof} We have $s \in \bG^{*F^*}$ is a semisimple element which is real in $\bG^{*F^*}$, so let $h \in \bG^{*F^*}$ be such that $hsh^{-1} = s^{-1}$.  Since we assume that the center of $\bG$ is connected, $C_{\bG^*}(s)$ is a connected reductive group.  We may then define an automorphism $\sigma_h$ on the connected reductive group $C_{\bG^*}(s)$ by $\sigma_h(a) = hah^{-1}$, and if $s$ is not an involution, then $\sigma_h$ is an outer automorphism.  Since $h \in \bG^{*F^*}$, $\sigma_h$ commutes with $F^*$, that is, $\sigma_h \in \Out(C_{\bG^*}(s), F^*)$.  In particular, if $C_{\bG^*}(s)$ has no two simple factors which are isogenous,  no simple factor of type $D_{2n}$, none of type $C_2$ or $F_4$ if $p=2$, and none of type $G_2$ if $p=3$, then Lemma \ref{GalAct} applies to the automorphism $\sigma_h$ and so $\psi$ is invariant under $\sigma_h$.  In this situation (or if $s$ is an involution) it follows from Theorem \ref{RealJord} and Lemma \ref{GalAct} that $\chi$ is real-valued if and only if $\psi$ is real-valued.
\end{proof}

\section{The groups $E_7(q)_{\ad}$ and $E_8(q)$} \label{E7E8Section}

In this section we give our main results for the exceptional groups $E_7(q)_{\ad}$ and $E_8(q)$.  We begin by noticing that since $-1$ is in the Weyl group of type $E_7$ and type $E_8$, by a result of Singh and Thakur \cite[Theorem 2.3.1]{SiTh08} every semisimple element $s$ in each of the groups $\bG^F = E_7(q)_{\ad}, E_7(q)_{\msc},$ or $E_8(q)$ is real in $\bG^F$.  It follows from Theorem \ref{RealJord} that for every irreducible character $\chi$ of $\bG^F$ the characters $\chi$ and $\bar{\chi}$ are in the same Lusztig series $\cE(\bG^F, (s))$.

The notation for the relevant unipotent characters of the groups in the results of this section is slightly adjusted from that of \cite[pp.\ 483--488]{Ca85} for the sake of clarity.  For example, instead of using the notation $E_6[\theta^2], \epsilon$ for the unipotent character in the last line of \cite[p.~483]{Ca85}, we will write $E_6[\theta^2, \epsilon]$.  The structure of centralizers of semisimple elements in these results is given as a product of the simple factor types along with any central torus factor which occurs, where a polynomial in $q$ denotes a torus of that order.  We now give our results for the group $E_7(q)_{\ad}$.

\begin{theorem} \label{E7} Let $q$ be a prime power.  The number of irreducible characters of the group $\bG^F =E_7(q)_{\ad}$ which are not real-valued is $2q+4$ if $q$ is even, and $2q+8$ if $q$ is odd.  These characters which are not real-valued are:
\begin{itemize}
\item The unipotent characters $E_7[\xi]$, $E_7[-\xi]$, $E_6[\theta, 1]$, $E_6[\theta^2, 1]$, $E_6[\theta, \epsilon]$, and $E_6[\theta^2, \epsilon]$;
\item When $q$ is odd, the six unipotent characters above tensored with the linear character of order $2$ of $\bG^F$;
\item A conjugate pair of characters in each of $\frac{1}{2} (q-2)$  (if $q$ is even) or $\frac{1}{2}(q-3)$ (if $q$ is odd) Lusztig series $\cE(\bG^F, s)$, where $(s)$ is a semisimple class in $\bG^{*F^*}$ such that $C_{\bG^*}(s)^{F^*}$ is of type $E_6(q).(q-1)$;
\item A conjugate pair of characters in each of $\frac{1}{2} q$  (if $q$ is even) or $\frac{1}{2}(q-1)$ (if $q$ is odd) Lusztig series $\cE(\bG^F, s)$, where $(s)$ is a semisimple class in $\bG^{*F^*}$ such that $C_{\bG^*}(s)^{F^*}$ is of type ${^2E_6(q)}.(q+1)$.
\end{itemize}
Moreover, every real-valued irreducible characters of $\bG^F$ has Frobenius--Schur indicator $1$, that is, $m_{\R}(\chi) = 1$ for all $\chi \in \Irr(\bG^F)$.
\end{theorem}
\begin{proof} We first explain why the list of characters are not real-valued.  It follows from \cite[Table 1]{Ge03} that the cuspidal unipotent characters $E_7[\xi]$ and $E_7[-\xi]$ of $E_7(q)_{\ad}$ (or of $E_7(q)_{\mathrm{sc}}$) are not real-valued for any $q$.  The unipotent characters $E_6[\theta, 1]$, $E_6[\theta, \epsilon]$, $E_6[\theta^2, 1]$, and $E_6[\theta^2, \epsilon]$ are not cuspidal, and are constituents of a Harish-Chandra series corresponding to a Levi component which is type $E_6(q)$, parabolically induced from the cuspidal unipotent character $E_6[\theta]$ of $E_6(q)$ for the first two cases, and from the cuspidal unipotent character $E_6[\theta^2]$ in the last two cases.  The cuspidal unipotent characters $E_6[\theta]$ and $E_6[\theta^2]$ are not real-valued again by \cite[Table 1]{Ge03}.  It follows from \cite[Proposition 5.6]{Ge03} that the fields of character values of $E_6[\theta, 1]$ and $E_6[\theta, \epsilon]$ are the same as that of $E_6[\theta]$, and the fields of character values of $E_6[\theta^2, 1]$ and $E_6[\theta^2, \epsilon]$ are the same as that of $E_6[\theta^2]$, and so these unipotent characters of $E_7(q)_{\ad}$ are not real-valued.  When $q$ is odd, the derived subgroup of $\bG^F$ is the simple group $E_7(q)$, and is an index $2$ subgroup of $\bG^F$.  Thus $\bG^F$ has a linear character $\lambda$ of order $2$ when $q$ is odd. Tensoring the six unipotent characters which are not real-valued of $\bG^F$ with $\lambda$ produces six distinct irreducible characters which are not unipotent by \cite[Proposition  13.30(ii)]{dmbook}, and does not change the field of values. Thus these six characters are also not real-valued when $q$ is odd.

Next, it follows from \cite{De83, FlJa93} that the group $\bG^{*F^*} = E_7(q)_{\mathrm{sc}}$ has $\frac{1}{2} (q-2)$ (if $q$ is even) or $\frac{1}{2}(q-3)$ (if $q$ is odd) semisimple classes $(s)$ such that $C_{\bG^*}(s)^{F^*}$ is of type $E_6(q).(q-1)$, and has $\frac{1}{2} q$ (if $q$ is even) or $\frac{1}{2} (q-1)$ (if $q$ is odd) semisimple classes $(s)$ such that $C_{\bG^*}(s)^{F^*}$ is of type ${^2E_6(q)}.(q+1)$.  Since the central torus factor of these centralizers (cyclic of order $q \pm 1$) has only the trivial character as a unipotent character, then as described in the proof of Lemma \ref{GalAct} above, the centralizer of type $E_6(q).(q-1)$ has unipotent characters with the same character values as the cuspidal unipotent characters $E_6[\theta]$ and $E_6[\theta^2]$ of the $E_6(q)$ factor.  These cuspidal unipotent characters of $E_6(q)$ are not real-valued by \cite[Table 1]{Ge03}, and these are in fact complex conjugates.  Since the semisimple classes $(s)$ are all real, it follows from Lemma \ref{RealLemma} that the characters of $E_7(q)_{\ad}$ with Jordan decomposition $(s, E_6[\theta])$ or $(s, E_6[\theta^2])$ are not real-valued, and that these are conjugate pairs of characters of $E_7(q)_{\ad}$ by Theorem \ref{RealJord}, and there exists one such pair for each corresponding semisimple class $(s)$ in $\bG^{*F^*}$.  By the same argument, the centralizers of the semisimple classes $(s)$ in $\bG^{*F^*}$ of type ${^2E_6(q)}.(q+1)$ have unipotent characters with the same character values as the cuspidal unipotent characters ${^2E_6}[\theta]$ and ${^2E_6}[\theta^2]$ of ${^2E_6}(q)$, which are not real-valued by \cite[Table 1]{Ge03} and which are conjugate pairs.  Again by Lemma \ref{RealLemma} and Theorem \ref{RealJord}, the characters of $E_7(q)_{\ad}$ with Jordan decomposition $(s, {^2E_6}[\theta])$ and $(s, {^2E_6}[\theta^2])$ are not real-valued and are conjugate pairs, as $(s)$ varies over these semisimple classes of $\bG^{*F^*}$. This shows that the characters listed are all not real-valued.

We now prove that the remaining irreducible characters of $E_7(q)_{\ad}$ are real-valued, and they all have Frobenius--Schur indicator $1$, using the method employed in Sections \ref{ExamplesSection} and \ref{F4Section}.  The number of involutions in $E_7(q)_{\ad}$ may be computed in the case that $q$ is even using \cite{AsSe76}, and in the case that $q$ is odd using any of \cite{Iw70, De83, FlJa93}, and the classification of involutions is also nicely summarized in \cite{BuTh18}.  The number of involutions obtained as a polynomial in $q$ is given in Table \ref{InvolTable}.  The total character degree sum for $E_7(q)_{\ad}$ may be computed directly using the data of L\"{u}beck \cite{LuWWW}, and the result as a polynomial in $q$ is given in Table \ref{SumTable}.  The degrees of the characters of $E_7(q)_{\ad}$ which are not real-valued may be obtained as follows.  The degrees of the unipotent characters (and their tensors with the linear character of order 2) are given in \cite[p.~483]{Ca85} or in \cite{Lu84}.  As in Section \ref{JDSection}, the degree of any character with Jordan decomposition $(s, \psi)$ is $[\bG^{*F^*}: C_{\bG^{*}}(s)^{F^*}]_{p'} \psi(1)$.  The degrees $\psi(1)$ of the relevant unipotent characters of $E_6(q)$ and ${^2 E_6}(q)$ are given in \cite[pp.\ 480-481]{Ca85}.  The resulting degrees of the characters of $E_7(q)_{\ad}$ which are not real-valued are given in Table \ref{E7Table}.  Summing these degrees with the number of involutions, whether $q$ is even or odd, we obtain the total character degree sum, which implies the result by the Frobenius--Schur involution count.  \end{proof}

The following is our result for the group $E_8(q)$.  The strategy used is that in the proof of Theorem \ref{E7}, although the list of characters which are not real-valued is significantly longer for $E_8(q)$, and so the description of these characters is given in the proof of Theorem \ref{E8} and in Table \ref{E8Table}.

\begin{theorem} \label{E8} Let $q$ be a prime power.  The number of irreducible characters of the group $\bG^F = E_8(q)$ which are not real-valued is $2q^2 + 6q + 20$ if $q$ is a power of $2$, $2q^2 + 6q + 22$ if $q$ is a power of $3$, and $2q^2 + 6q+24$ otherwise.  The description of these characters in terms of Jordan decomposition is given below and in Table \ref{E8Table}.  In particular, every real-valued irreducible character of $\bG^F$ has Frobenius--Schur indicator $1$, that is, $m_{\R}(\chi) = 1$ for all $\chi \in \Irr(\bG^F)$.
\end{theorem}

\begin{proof}
First, it follows from \cite[Table 1]{Ge03} that the ten cuspidal unipotent characters $E_8[\theta]$, $E_8[-\theta]$, $E_8[\theta^2]$, $E_8[-\theta^2]$, $E_8[\zeta]$, $E_8[\zeta^2]$, $E_8[\zeta^3]$, $E_8[\zeta^4]$, $E_8[i]$, and $E_8[-i]$ of $E_8(q)$ are not real-valued for any $q$. The unipotent characters $E_7[\xi, 1]$, $E_7[-\xi,1]$, $E_7[\xi, \epsilon]$, $E_7[-\xi, \epsilon]$, $E_6[\theta, \phi_{1,0}]$, $E_6[\theta^2, \phi_{1,0}]$, $E_6[\theta, \phi_{1,3}']$, $E_6[\theta^2, \phi_{1,3}']$, $E_6[\theta, \phi_{1,3}'']$, $E_6[\theta^2, \phi_{1,3}'']$, $E_6[\theta, \phi_{1,6}]$, $E_6[\theta^2, \phi_{1,6}]$, $E_6[\theta, \phi_{2,1}]$, $E_6[\theta^2, \phi_{2,1}]$, $E_6[\theta, \phi_{2,2}]$, and $E_6[\theta^2, \phi_{2,2}]$ are Harish-Chandra induced from the non-real cuspidal unipotent characters $E_7[\xi]$, $E_7[-\xi]$, $E_6[\theta]$, and $E_6[\theta^2]$ of the Levi components of types $E_7(q)$ and $E_6(q)$.  As argued in the proof of Theorem \ref{E7}, it follows from \cite[Proposition 5.6]{Ge03} that these are not real-valued.

From \cite{De83, FlJa94}, we see that $\bG^{*F^*} = E_8(q)$ has semisimple classes $(s)$ such that $C_{\bG^*}(s)^{F^*}$ is one of the following types: $E_7(q).(q- 1)$, $E_7(q).(q+1)$, $E_6(q).(q-1)^2$, $E_6(q).(q^2-1),$ $E_6(q).(q^2+q+1)$, ${^2 E_6}(q).(q+1)^2$, ${^2 E_6}(q).(q^2-1)$, or ${^2 E_6}(q).(q^2-q+1)$. The number of classes of each centralizer type depends on the residue class of $q$ modulo 6, and is provided in \cite{De83, FlJa94} or in L\"ubeck's data on centralizers of semisimple elements \cite{LuWWW1}, but can be inferred by dividing entries from the last five columns of Table \ref{E8Table} by 2. In each of these cases, the central torus factor has only the trivial character as a unipotent character, and therefore, as in the proof of Lemma \ref{GalAct}, $C_{\bG^*}(s)^{F^*}$ has unipotent characters that have the same character values as the cuspidal unipotent characters $E_7[\xi]$ and $E_7[-\xi]$, $E_6[\theta]$ and $E_6[\theta^2]$, or ${^2 E_6}[\theta]$ and ${^2 E_6}[\theta^2]$ of the $E_7(q),$ $E_6(q)$, or ${^2 E_6}(q)$ factors. By \cite[Table 1]{Ge03}, these cuspidal unipotents are not real-valued, and since the semisimple classes $(s)$ are all real classes, it follows from Lemma \ref{RealLemma} that the characters of $E_8(q)$ with Jordan decomposition $(s, E_7[\xi])$ or $(s, E_7[-\xi])$, $(s, E_6[\theta^2])$ or $(s, E_6[\theta])$, and $(s,{^2 E_6}[\theta])$ or $(s, {^2 E_6}[\theta^2])$ are not real-valued, and are complex conjugate pairs (respectively) as $(s)$ varies over these semisimple classes of $\bG^{*F^*}$. 

Next, $\bG^{*F^*}$ has semisimple classes $(s)$ such that $C_{\bG^*}(s)^{F^*}$ is one of the following types: $E_7(q).A_1(q)$, $E_6(q).A_1(q).(q-1)$, or ${^2 E_6}(q).A_1(q).(q+1)$. As before, the number of classes of each centralizer type can be obtained by dividing entries of the last five columns of Table \ref{E8Table} by 2. The central torus factor has only the trivial character as a unipotent character, and in addition to the trivial character, the $A_1(q)$ factor has a unipotent character $\nu$ of degree $q$ (see \cite[p.~465]{Ca85}). Therefore, as in the proof of Lemma \ref{GalAct}, $C_{\bG^*}(s)^{F^*}$ has unipotent characters that have the same character values as $E_7[\xi]$, $E_7[-\xi]$, $E_7[\xi]\otimes \nu$, and $E_7[-\xi]\otimes \nu$, $E_6[\theta]$, $E_6[\theta^2]$, $E_6[\theta]\otimes \nu$, and $E_6[\theta^2]\otimes\nu$, or ${^2 E_6}[\theta]$, ${^2 E_6}[\theta^2]$, ${^2 E_6}[\theta]\otimes\nu$, and ${^2 E_6}[\theta^2]\otimes\nu$, none of which is real-valued.  It follows from Lemma \ref{RealLemma} that the characters of $E_8(q)$ with Jordan decomposition $(s, E_7[\xi])$, $(s, E_7[-\xi])$, $(s, E_7[\xi]\otimes \nu)$, or $(s, E_7[-\xi]\otimes \nu)$, $(s, E_6[\theta])$, $(s, E_6[\theta])$, $(s, E_6[\theta]\otimes \nu)$, or $(s, E_6[\theta]\otimes \nu)$, or $(s,{^2 E_6}[\theta])$, $(s, {^2 E_6}[\theta^2])$, $(s,{^2 E_6}[\theta]\otimes \nu)$, or $(s, {^2 E_6}[\theta^2]\otimes \nu)$ are not real-valued, and are complex conjugate pairs as $(s)$ varies over these semisimple classes of $\bG^{*F^*}$.

Finally, when $q\equiv 1,4 \bmod 6$, $\bG^{*F^*}$ has a semisimple class $(s)$ such that $C_{\bG^*}(s)^{F^*}$ is of type $E_6(q).A_2(q)$, and when $q\equiv 2,5 \bmod 6$, $\bG^{*F^*}$ has a semisimple class $(s)$ such that $C_{\bG^*}(s)^{F^*}$ is of type ${^2 E_6}(q).{^2 A_2}(q)$. The $A_2(q)$ factor has unipotent characters of degree $1$, $q(q+1)$, and $q^3$, and the ${^2 A_2}(q)$ factor has unipotent characters of degree $1$, $q(q-1)$, and $q^3$ (see \cite[p.~465]{Ca85}). Therefore, $C_{\bG^*}(s)^{F^*}$ has unipotent characters that have the same character values as the cuspidal unipotent characters $E_6[\theta]$ and $E_6[\theta^2]$, or ${^2 E_6}[\theta]$ and ${^2 E_6}[\theta^2]$, as well as the tensor product of these characters with the other unipotent characters of $A_2(q)$ or ${^2 A_2}(q)$, each of which are characters which are not real-valued. As above, the resulting six characters of $E_8(q)$ (provided $q\not\equiv 3\bmod 6$) are not real-valued.

To complete the proof, we observe that the sum of the degrees of the aforementioned characters that are not real-valued, given in the fourth column of Table \ref{E8Table}, together with the number of involutions provided in Table \ref{InvolTable}, is equal to the total sum of the degrees of all irreducible characters of $E_8(q)$, which is given in Table \ref{SumTable}.  From the Frobenius--Schur involution count, it follows that $\ep(\chi)=1$ for all real-valued irreducible characters $\chi$ of $E_8(q)$.  The total number of characters which are not real-valued is obtained by summing the entries in the last five columns of Table \ref{E8Table}.
\end{proof}

\section{Remarks on remaining cases} \label{RemarksSec}

We expect that a statement similar to Theorem \ref{E7} holds for the simple group $E_7(q)$.  If $E_7(q)_{\msc}$ is distinct from $E_7(q)_{\ad}$ (when $q$ is odd), then for $\bG^F = E_7(q)_{\msc}$, the group $\bG$ has disconnected center of order 2, and $E_7(q)$ is the quotient of $E_7(q)_{\msc}$ by its center.  Then the results in Theorem \ref{RealJord} and Lemma \ref{RealLemma} are not known to hold in this case, which is one obstruction.  As mentioned in the Remark at the end of \cite{SV15}, we also expect that Theorem \ref{RealJord} holds more generally than when $Z(\bG)$ is connected, namely it should hold for Lusztig series in cases when $C_{\bG^*}(s)$ is connected (while $Z(\bG)$ is not necessarily connected).  However, it seems that also $E_7(q)_{\msc}$ has characters which are not real-valued in Lusztig series in cases when $C_{\bG^*}(s)$ is disconnected with two components, but more information on the action on Jordan decomposition in this case is needed.  It appears from numerical investigation that for $E_7(q)_{\msc}$ with $q$ odd that a real-valued character has Frobenius--Schur indicator $1$ if and only if its central character is trivial.  To prove this, one needs to show that the sum of the degrees of the real-valued irreducible characters of $E_7(q)_{\msc}$ is equal to the number of elements which square to the non-trivial central element.  Once this result is proved, it will follow for the simple group $E_7(q)$ that every real-valued irreducible character of $E_7(q)$ will have Frobenius--Schur indicator 1, since $E_7(q)$ is a quotient of $E_7(q)_{\msc}$ by its center.  We note that our results for $E_7(q)_{\ad}$ are not enough to draw these conclusions for the index 2 simple group $E_7(q)$, since we do not have enough information to rule out the possibility that an irreducible character of $E_7(q)_{\ad}$ which is not real-valued could restrict to $E_7(q)$ to give a character with a component with Frobenius--Schur indicator $-1$.

As in the beginning of the previous section, it is helpful that every semisimple class of the groups $E_7(q)_{\msc}$, $E_7(q)_{\ad}$, and $E_8(q)$ is real.  This is not true in the groups $E_6(q)_{\ad}$, $E_6(q)_{\msc}$, ${^2 E_6}(q)_{\ad}$, or ${^2 E_6}(q)_{\msc}$.  So a first necessary step in understanding these cases is to classify the semisimple classes in these groups which are real.  Additionally, it appears that there are some real semsimple classes $(s)$ of $\bG^{*F^*} =E_6(q)_{\msc}$ such that the Lusztig series $\cE(\bG^F, s)$ has characters which are not real-valued, while $C_{\bG^*}(s)$ has a factor of type $D_4$.  In particular, as in \cite[Lemma 1.64]{JT12} there are unipotent characters of groups of type $D_{2n}$ which are not invariant under the order 2 graph automorphism.  Then our Lemma \ref{RealLemma} does not apply, and one needs more specific information on the action of the inverting element of the real semisimple element $s$ in order to apply Theorem \ref{RealJord} and prove specific characters are not real-valued.  Also, in the cases that $E_6(q)_{\msc}$ has a center of order 3, there are Lusztig series such that $C_{\bG^*}(s)$ is disconnected with 3 components.  It seems plausible that the fact that the number of components is odd will be enough for the conclusions of Lemma \ref{RealLemma} to still hold.  Finally, it has been proved by Ohmori \cite{Oh93} that the group ${^2 E_6}(q)$ has at least two characters which have Frobenius--Schur indicator $-1$.  It may be observed by its character table that the group ${^2 E_6}(2)$ has 3 irreducible characters with this property, and so we must also understand the total number of such characters in the general case.  We hope to address all of these issues in a subsequent paper, and complete the problem of determining the Frobenius--Schur indicators of all irreducible characters of the finite exceptional groups.  

\newpage
\appendix
\section{Tables} \label{Tables}
In Table \ref{G2CMC}, we list some class multiplication coefficients for the group $G_2(q)$ where $q$ is even or a power of 3, proving that every real class is strongly real. When $q$ is even, $G_2(q)$ has two classes of involutions, labeled as {\em class type 2} and {\em class type 3} in {\sf CHEVIE}, and two non-real classes, {\em class type 12} and {\em class type 13}. When $q$ is a power of 3, $G_2(q)$ has one class of involutions, {\em class type 10}, and two non-real classes, {\em class type 8} and {\em class type 9}. To compute the class multiplication coefficient $n_{123}$ for the group $G$, where the conjugacy classes $C_1, C_2, C_3$ are labeled as {\em class type i, class type j, class type k} in {\sf CHEVIE}, we use the sequence of commands \begin{verbatim}> GenCharTab(G);
> ClassMult(g,i,j,k);\end{verbatim} As all class multiplication coefficients are positive for all relevant $q$ in the second, third, and fifth columns, every real element is a product of two involutions when $q$ is even or a power of 3. 
\setcounter{table}{0} \renewcommand{\thetable}{A.\arabic{table}}
\begin{table}[h!]\caption{Class Multiplication Coefficients of $G_2(q)$ }\label{G2CMC}
\resizebox{\linewidth}{!}{
\begin{tabular}{|c|c|c|c|c|}\hline
Class Triple  & ClassMult\_G$2.01(i,j,k)$ & ClassMult\_G$2.02(i,j,k)$ & Class Triple & ClassMult\_G$2.10(i,j,k)$ \\ 
 $(i,j,k)$ & $q$ even, $q\equiv 1\bmod 3$ & $q$ even, $q\equiv 2\bmod 3$ & $(i,j,k)$ & $q\equiv 0 \bmod 3$  \\ \hline
$(3,3,1)$ & $q^8-q^2$ & $q^8-q^2$ & $(10,10,1)$ & $q^8+q^6+q^4$ \\ \hline
$(3,3,2)$ & $q^5+q^4-q^3-q^2$  & $q^5+q^4-q^3-q^2$ & $(10,10,2)$ & $q^5+q^4$ \\ \hline
$(2,3,3)$ & $q^3+q^2-q-1$ & $q^3+q^2-q-1$ & $(10,10,3)$ & $q^5+q^4$ \\ \hline
$(3,3,4)$ & $4q^3-q^2+6q$ & $4q^3-q^2+4q$ & $(10,10,4)$ & $q^4$\\ \hline
$(3,3,5)$ & $2q^3-q^2$ & $2q^3-q^2+2q$ & $(10,10,5)$ & $2q^3$ \\ \hline
$(2,3,6)$ & $q^2-q$ & $q^2+q$ & $(10,10,6)$ & $2q^3$ \\ \hline
$(2,3,7)$ & $2q$ & $2q$ & $(10,10,7)$ & $3q^2$ \\ \hline
$(3,3,8)$ & $2q^2$ & $2q^2$ & $(10,10,10)$ & $q^4+q^2$ \\ \hline
$(3,3,9)$ & $q^5-q^2$ & $q^5+q^2$ & $(10,10,11)$ & $q^3+q^2$ \\ \hline
$(3,3,10)$ & $q^3-q^2$ & $q^3+q^2$ & $(10,10,12)$ & $q^3+q^2$ \\ \hline
$(3,3,11)$ & $3q^2$ & $3q^2$ & $(10,10,13)$ & $2q^2$ \\ \hline
$(3,3,14)$ & $q^3-q^2$ & $q^3-q^2$ & $(10,10,14)$ & $2q^2$ \\ \hline
$(3,3,15)$ & $q^2-q$ & $q^2-q$ & $(10,10,15)$ & $q^3-q^2$ \\ \hline
$(3,3,16)$ & $q^3-q^2-q+1$ & $q^3-q^2-q+1$ & $(10,10,16)$ & $q^2-q$ \\ \hline
$(2,3,17)$ & $q-1$ & $q-1$ & $(10,10,17)$ & $q^3-q^2$ \\ \hline
$(3,3,18)$ & $q^2-2q+1$ & $q^2-2q+1$ & $(10,10,18)$ & $q^2-q$ \\ \hline
$(3,3,19)$ & $q^3+q^2$ & $q^3+q^2$ & $(10,10,19)$ & $q^3+q^2$ \\ \hline
$(3,3,20)$ & $q^2+q$ & $q^2+q$ & $(10,10,20)$ & $q^2+q$ \\ \hline
$(3,3,21)$ & $q^3+q^2-q-1$ & $q^3+q^2-q-1$ & $(10,10,21)$ & $q^3+q^2$ \\ \hline
$(2,3,22)$ & $q+1$ & $q+1$ & $(10,10,22)$ & $q^2+q$ \\ \hline
$(3,3,23)$ & $q^2+2q+1$ & $q^2+2q+1$ & $(10,10,23)$ & $q^2-2q+1$ \\ \hline
$(3,3,24)$ & $q^2-1$ & $q^2-1$ & $(10,10,24)$ & $q^2-1$ \\ \hline
$(3,3,25)$ & $q^2-1$ & $q^2-1$ & $(10,10,25)$ & $q^2-1$ \\ \hline
$(3,3,26)$ & $q^2+q+1$ & $q^2+q+1$ & $(10,10,26)$ & $q^2+2q+1$ \\ \hline
$(3,3,27)$ & $q^2-q+1$ & $q^2-q+1$ & $(10,10,27)$ & $q^2+q+1$ \\ \hline

    \multicolumn{3}{c|}{} & $(10,10,28)$ & $q^2-q+1$ \\ \cline{4-5}
\end{tabular}}
\end{table}

\newpage

In Table \ref{SumTable}, the sums of the character degrees of the relevant exceptional groups are listed as polynomials in $q$.  These are all computed directly from the data of L\"{u}beck \cite{LuWWW}.  While the lists of character degrees given in \cite{LuWWW} for $E_7(q)_{\ad}$ depend on $q$ mod $12$, and for $E_8(q)$ depend on $q$ mod $60$, the character degree sums depend only on the parity of $q$.

\begin{center}
\begin{table}[h!]\caption{Character Degree Sums}\label{SumTable}
\resizebox{\linewidth}{!}{
\begin{tabular}{|c|c|}\hline

$F_4(q), q$ even & $q^{28}+q^{26}+q^{24}+q^{22}+\frac{1}{6}q^{20}-\frac{7}{3}q^{18}-\frac{1}{2}q^{16}-\frac{2}{3}q^{14}+\frac{2}{3}q^{12}-\frac{2}{3}q^{10}+\frac{3}{2}q^8-\frac{7}{3}q^6+\frac{7}{6}q^4$ \\  \hline

$F_4(q), q$ odd & $q^{28}+q^{26}+2q^{24}+2q^{22}+\frac{19}{6}q^{20}-\frac{1}{3}q^{18}+\frac{5}{2}q^{16}+\frac{4}{3}q^{14}+\frac{5}{3}q^{12}+\frac{1}{3}q^{10}+\frac{3}{2}q^8-\frac{7}{3}q^6+\frac{7}{6}q^4+1$ \\ \hline

 & $q^{28}-q^{26}+\sqrt{2}q^{25}+q^{24}-\sqrt{2}q^{23}+q^{22}-\sqrt{2}q^{21}+\frac{4}{3}q^{20}+3\sqrt{2}q^{19}$ \\ 
$^{2}F_4(q^2), q^2= 2^{2m+1} $ & $-2q^{18}-\sqrt{2}q^{17}+2q^{16}-3\sqrt{2}q^{15}-q^{14}+3\sqrt{2}q^{13}-\frac{8}{3}q^{12}+\sqrt{2}q^{11}$\\ 
 & $+q^{10}-3\sqrt{2}q^9+\sqrt{2}q^7-2q^6+\sqrt{2}q^5+\frac{7}{3}q^4-\sqrt{2}q^3$ \\ \hline
& $q^{70}+q^{66}+q^{64}+q^{62}+q^{60}+q^{58}+\frac{2}{3}q^{57}-\frac{4}{3}q^{55}+q^{54}-\frac{2}{3}q^{51}+q^{50}
-\frac{16}{3}q^{49}+5q^{48}$ \\ 
 $E_7(q)_{\ad}$,  & $-\frac{19}{3}q^{47}+8q^{46}-\frac{35}{3}q^{45}+14q^{44}-21q^{43}+23q^{42}
-\frac{76}{3}q^{41}+29q^{40}-\frac{101}{3}q^{39}+37q^{38}$ \\
$q$ even & $-\frac{131}{3}q^{37}+45q^{36}-\frac{134}{3}q^{35}
+48q^{34}-\frac{151}{3}q^{33}+51q^{32}-52q^{31}+51q^{30}-\frac{140}{3}q^{29}$ \\
 & $+47q^{28}
-\frac{133}{3}q^{27}+41q^{26}-\frac{115}{3}q^{25}+35q^{24}-\frac{86}{3}q^{23}+26q^{22}-24q^{21}
+19q^{20}-\frac{49}{3}q^{19}$ \\ 
& $+13q^{18}-\frac{28}{3}q^{17}+7q^{16}-6q^{15}+4q^{14}-2q^{13}
+2q^{12}-q^{11}$ \\ \hline
& $q^{70}+q^{66}+2q^{64}+3q^{62}+4q^{60}+6q^{58}+\frac{2}{3}q^{57}+7q^{56}-\frac{4}{3}q^{55}
+10q^{54}+12q^{52}-\frac{8}{3}q^{51}$ \\ $E_7(q)_{\ad}$, & $+15q^{50}-\frac{28}{3}q^{49}+24q^{48}-\frac{38}{3}q^{47}
+31q^{46}-26q^{45}+41q^{44}-40q^{43}+58q^{42}-50q^{41}$ \\ 
$q$ odd & $+68q^{40}-70q^{39}+81q^{38}-84q^{37}+95q^{36}-\frac{268}{3}q^{35}+99q^{34}-\frac{310}{3}q^{33}+103q^{32}
-\frac{304}{3}q^{31}$ \\ & $+103q^{30}-\frac{280}{3}q^{29}+93q^{28}-92q^{27}+82q^{26}-74q^{25}
+70q^{24}-58q^{23}+52q^{22}-50q^{21}$ \\ & $+38q^{20}-30q^{19}+26q^{18}-\frac{56}{3}q^{17}
+14q^{16}-\frac{40}{3}q^{15}+8q^{14}-\frac{8}{3}q^{13}+4q^{12}-2q^{11}-\frac{4}{3}q^9+\frac{2}{3}q^7+1$ \\ \hline
 & $q^{128}+q^{124}+q^{122}+q^{120}+q^{118}+q^{116}+\frac{2}{3}q^{115}-\frac{2}{3}q^{113}+2q^{112}
-\frac{2}{3}q^{111}+q^{110}+3q^{108}$ \\ 
& $-\frac{26}{3}q^{107}+10q^{106}-13q^{105}+\frac{203}{10}q^{104}
-24q^{103}+\frac{176}{5}q^{102}-\frac{173}{3}q^{101}+\frac{677}{10}q^{100}-\frac{248}{3}q^{99}$ \\ 
& $+108q^{98}
-\frac{398}{3}q^{97}+173q^{96}-225q^{95}+\frac{2561}{10}q^{94}-297q^{93}+\frac{3617}{10}q^{92}
-\frac{1259}{3}q^{91}+495q^{90}$ \\ 
& $-\frac{1798}{3}q^{89}+\frac{6653}{10}q^{88}-735q^{87}+\frac{4236}{5}q^{86}
-953q^{85}+\frac{5373}{5}q^{84}-1219q^{83}+\frac{6541}{5}q^{82}$ \\ 
& $-1408q^{81}+\frac{15681}{10}q^{80}
-1703q^{79}+\frac{9178}{5}q^{78}-\frac{6029}{3}q^{77}+\frac{10556}{5}q^{76}-2200q^{75}
+\frac{23681}{10}q^{74}$ \\ 
$E_8(q)$, & $-\frac{7522}{3}q^{73}+2626q^{72}-\frac{8309}{3}q^{71}+\frac{14169}{5}q^{70}
-\frac{8660}{3}q^{69}+\frac{15157}{5}q^{68}-\frac{9374}{3}q^{67}+\frac{15831}{5}q^{66}$ \\ 
$q$ even & $-3253q^{65}
+\frac{32683}{10}q^{64}-3236q^{63}+\frac{16513}{5}q^{62}-\frac{10006}{3}q^{61}+\frac{16474}{5}q^{60}
-\frac{9857}{3}q^{59}+\frac{16098}{5}q^{58}$ \\ 
& $-\frac{9352}{3}q^{57}+\frac{31093}{10}q^{56}-\frac{9184}{3}q^{55}
+\frac{14671}{5}q^{54}-2851q^{53}+\frac{13687}{5}q^{52}-2575q^{51}+\frac{12464}{5}q^{50}
$ \\ 
& $-\frac{7225}{3}q^{49}+2247q^{48}-\frac{6328}{3}q^{47}+\frac{19691}{10}q^{46}-\frac{5416}{3}q^{45}
+\frac{8536}{5}q^{44}-1601q^{43}+\frac{7178}{5}q^{42}$ \\ 
& $-1308q^{41}+\frac{11881}{10}q^{40}
-\frac{3148}{3}q^{39}+\frac{4786}{5}q^{38}-\frac{2632}{3}q^{37}+\frac{3788}{5}q^{36}-\frac{1966}{3}q^{35}
+\frac{2861}{5}q^{34}$ \\ 
& $-\frac{1457}{3}q^{33}+\frac{4323}{10}q^{32}-
\frac{1139}{3}q^{31}+303q^{30}
-\frac{749}{3}q^{29}+\frac{2117}{10}q^{28}-\frac{499}{3}q^{27}+\frac{1401}{10}q^{26}-\frac{359}{3}q^{25}
+$ \\ 
& $88q^{24}-64q^{23}+50q^{22}-35q^{21}+\frac{277}{10}q^{20}-\frac{70}{3}q^{19}+\frac{56}{5}q^{18}
-7q^{17}+\frac{73}{10}q^{16}-3q^{15}$ \\ 
& $+2q^{14}-\frac{7}{3}q^{13}+\frac{2}{3}q^{11}$ \\ \hline
 & $q^{128}+q^{124}+q^{122}+2q^{120}+2q^{118}+3q^{116}+\frac{2}{3}q^{115}+3q^{114}
-\frac{2}{3}q^{113}+6q^{112}-\frac{2}{3}q^{111}$ \\ 
& $+6q^{110}+11q^{108}-\frac{32}{3}q^{107}+20q^{106}
-19q^{105}+\frac{393}{10}q^{104}-\frac{109}{3}q^{103}+\frac{326}{5}q^{102}-\frac{259}{3}q^{101}
$ \\ 
& $+\frac{1167}{10}q^{100}-\frac{403}{3}q^{99}+189q^{98}-\frac{652}{3}q^{97}+298q^{96}-\frac{1099}{3}q^{95}
+\frac{4411}{10}q^{94}-508q^{93}+\frac{6347}{10}q^{92}$ \\ 
& $-\frac{2168}{3}q^{91}+872q^{90}-1031q^{89}
+\frac{11783}{10}q^{88}-\frac{3931}{3}q^{87}+\frac{7626}{5}q^{86}-\frac{5110}{3}q^{85}+\frac{9703}{5}q^{84}
$ \\ 
& $-2184q^{83}+\frac{11956}{5}q^{82}-2593q^{81}+\frac{28941}{10}q^{80}-3139q^{79}
+\frac{17063}{5}q^{78}-\frac{11158}{3}q^{77}+\frac{19796}{5}q^{76}$ \\ 
$E_8(q)$, & $-4167q^{75}+\frac{44851}{10}q^{74}
-\frac{14248}{3}q^{73}+5002q^{72}-\frac{15821}{3}q^{71}+\frac{27299}{5}q^{70}-\frac{16837}{3}q^{69}
+\frac{29382}{5}q^{68}$ \\ 
$q$ odd  & $-\frac{18200}{3}q^{67}+\frac{30966}{5}q^{66}-\frac{19079}{3}q^{65}+\frac{64433}{10}q^{64}
-\frac{19357}{3}q^{63}+\frac{32833}{5}q^{62}-\frac{19900}{3}q^{61}+\frac{32999}{5}q^{60}$ \\ 
& $-6589q^{59}
+\frac{32568}{5}q^{58}-6377q^{57}+\frac{63313}{10}q^{56}-\frac{18716}{3}q^{55}+\frac{30191}{5}q^{54}
-\frac{17615}{3}q^{53}+\frac{28387}{5}q^{52}$ \\ 
& $-\frac{16246}{3}q^{51}+\frac{26119}{5}q^{50}-\frac{15100}{3}q^{49}
+4747q^{48}-\frac{13421}{3}q^{47}+\frac{42111}{10}q^{46}-3918q^{45}+\frac{18391}{5}q^{44}
$ \\ 
& $-3441q^{43}+\frac{15703}{5}q^{42}-\frac{8609}{3}q^{41}+\frac{26261}{10}q^{40}-\frac{7102}{3}q^{39}
+\frac{10721}{5}q^{38}-\frac{5842}{3}q^{37}+\frac{8593}{5}q^{36}$ \\ 
& $-\frac{4510}{3}q^{35}+\frac{6626}{5}q^{34}
-\frac{3463}{3}q^{33}+\frac{10083}{10}q^{32}-\frac{2624}{3}q^{31}+732q^{30}-\frac{1832}{3}q^{29}
+\frac{5167}{10}q^{28}$ \\ 
& $-\frac{1282}{3}q^{27}+\frac{3511}{10}q^{26}-288q^{25}+228q^{24}
-\frac{523}{3}q^{23}+136q^{22}-\frac{317}{3}q^{21}+\frac{787}{10}q^{20}-\frac{175}{3}q^{19}$ \\ 
& $+\frac{196}{5}q^{18}
-27q^{17}+\frac{203}{10}q^{16}-\frac{40}{3}q^{15}+8q^{14}-4q^{13}+2q^{12}-q^{11}-\frac{2}{3}q^9+\frac{2}{3}q^7+1$ \\ \hline
\end{tabular}}
\end{table}
\end{center}

\newpage

In Table \ref{InvolTable}, we list the number of involutions in each group as a polynomial in $q$.  These can be computed when $q$ is even using \cite{AsSe76}, and when $q$ is odd using \cite{Iw70}.  These results are also summarized in \cite{BuTh18}.


\begin{center}
\begin{table}[h!]\caption{Numbers of Involutions}\label{InvolTable} 
\resizebox{\linewidth}{!}{
\begin{tabular}{|c|c|}\hline
$E_7(q)_{\ad}$, $q$ even & $q^{70}+q^{66}+q^{64}+q^{62}+q^{60}+q^{58}+q^{54}-q^{52}$ \\
 & $-q^{50}-q^{48}-2q^{46}-q^{44}-q^{42}-q^{40}+q^{32}+q^{28}$ \\ \hline
$E_7(q)_{\ad}$, $q$ odd & $q^{70}+q^{66}+2q^{64}+3q^{62}+4q^{60}+6q^{58}+7q^{56}+10q^{54}+10q^{52}+11q^{50}+12q^{48}$ \\
 & $+11q^{46}+11q^{44}+10q^{42}+8q^{40}+7q^{38}+5q^{36}+3q^{34}+3q^{32}+q^{30}+q^{28}+1$ \\ \hline
$E_8(q)$, $q$ even & $q^{128}+q^{124}+q^{122}+q^{120}+q^{118}+q^{116}+2q^{112}+q^{106}-q^{104}$ \\ 
 & $-2q^{98}-q^{94}-2q^{92}-2q^{88}-q^{86}-2q^{82}-q^{76}+q^{74}+q^{68}+q^{62}+q^{58}$ \\ \hline
 & $q^{128}+q^{124}+q^{122}+2q^{120}+2q^{118}+3q^{116}+3q^{114}+6q^{112}+5q^{110}$ \\
$E_8(q)$, $q$ odd  & $+7q^{108}+8q^{106}+9q^{104}+10q^{102}+12q^{100}+11q^{98}+14q^{96}+13q^{94}$ \\
 & $+14q^{92}+14q^{90}+14q^{88}+13q^{86}+14q^{84}+11q^{82}+12q^{80}+10q^{78}$ \\ 
 & $+9q^{76}+8q^{74}+7q^{72}+5q^{70}+5q^{68}+3q^{66}+3q^{64}+2q^{62}+q^{60}+q^{58}+q^{56}+1$ \\ \hline
\end{tabular}}
\end{table}
\end{center}

In Tables \ref{E7Table} and \ref{E8Table}, we list the degrees of the non-real characters of the groups $E_7(q)_{\ad}$ and $E_8(q)$. We write these degrees in terms of the cyclotomic polynomials $\Phi_n$, defined recursively by $\Phi_1=q-1$ and $$\Phi_n=\frac{q^n-1}{\displaystyle\prod_{i\mid n, 1\leq i<n}\Phi_i}.$$
As described in Section \ref{E7E8Section}, these non-real characters $\chi$ have a Jordan decomposition $(s, \psi)$, where $\psi$ is a (non-real) character of $C_{{\bf{G}^*}^{F^*}}(s)$ for some semisimple $s\in {\bf{G}^*}^{F^*}$. In the first two columns, we list these centralizers and the $p'$-part of their index in ${\bf{G}^*}^{F^*}$.  The structure of the centralizer in the first column is given in terms of the types of the simple algebraic group factors, along with any central torus factor, where a polynomial in $q$ represents a torus of that order.  In the third column, we list the degrees $\psi(1)$, and as $\chi(1)=[{\bf{G}^*}^{F^*}:C_{{\bf{G}^*}^{F^*}}(s)]_{p'}\psi(1)$, we obtain the degrees $\chi(1)$ listed in the fourth column. In the remaining columns, we list the number of non-real characters of each degree, which depends on $q$.

Note that in Table \ref{E8Table}, while the number of non-real characters of each degree depends on $q$ mod $6$, the total number of non-real characters of $E_8(q)$ depends only on whether $q$ is a power of $2$, a power of $3$, or a power of some other prime.
\begin{center}
\begin{table}[h!]\caption{Non-real characters of $E_7(q)_{\ad}$}\label{E7Table} 
\resizebox{\linewidth}{!}{
\begin{tabular}{|c|c|c|c|c|c|}\hline
Centralizer $C_{{\bf{G}^*}^{F^*}}(s)$ & $[{\bf{G}^*}^{F^*}:C_{{\bf{G}^*}^{F^*}}(s)]_{p'}$ & Degree of non-real unipotent in $C_{{\bf{G}^*}^{F^*}}(s)$ & Degree of non-real in $E_7(q)_{\ad}$ & $q$ even & $q$ odd \\ \hline
$E_7(q)$ & $1$ & $\frac{1}{2}q^{11}\Phi_5\Phi_9\Phi_{12}\Phi_8\Phi_7\Phi_4^2\Phi_3^3\Phi_1^7$ & $\frac{1}{2}q^{11}\Phi_5\Phi_9\Phi_{12}\Phi_8\Phi_7\Phi_4^2\Phi_3^3\Phi_1^7$ & 2 & 4 \\ \hline
$E_7(q)$ & $1$ & $\frac{1}{3}q^7\Phi_8\Phi_{10}\Phi_5\Phi_{14}\Phi_7\Phi_4^2\Phi_1^6\Phi_2^6$ & $\frac{1}{3}q^7\Phi_8\Phi_{10}\Phi_5\Phi_{14}\Phi_7\Phi_4^2\Phi_1^6\Phi_2^6$ & 2 & 4 \\ \hline 
$E_7(q)$ & $1$ & $\frac{1}{3}q^{16}\Phi_8\Phi_5\Phi_{10}\Phi_7\Phi_{14}\Phi_4^2\Phi_1^6\Phi_2^6$ & $\frac{1}{3}q^{16}\Phi_8\Phi_5\Phi_{10}\Phi_7\Phi_{14}\Phi_4^2\Phi_1^6\Phi_2^6$ & 2 & 4 \\ \hline 
$E_6(q).(q-1)$ & $\Phi_2^3\Phi_6\Phi_7\Phi_{10}\Phi_{14}\Phi_{18}$ & 
$\frac{1}{3}q^7\Phi_1^6\Phi_2^4\Phi_4^2\Phi_5\Phi_8$ & 
$\frac{1}{3}q^7\Phi_1^6\Phi_2^7\Phi_4^2\Phi_5\Phi_6\Phi_7\Phi_8\Phi_{10}\Phi_{14}\Phi_{18}$
 & $q-2$ & $q-3$\\ \hline
$^2E_6(q).(q+1)$ & $\Phi_1^3\Phi_3\Phi_5\Phi_7\Phi_9\Phi_{14}$ & $\frac{1}{3}q^7\Phi_1^4\Phi_2^6\Phi_4^2\Phi_8\Phi_{10}$ & $\frac{1}{3}q^7\Phi_1^7\Phi_2^6\Phi_3\Phi_4^2\Phi_5\Phi_7\Phi_8\Phi_9\Phi_{10}\Phi_{14}$ & $q$ & $q-1$\\ \hline
\end{tabular}}
\end{table}
\end{center}
\newpage
\begin{center}
\begin{table}[h!]\caption{Non-real characters of $E_8(q)$}\label{E8Table} 
\begin{adjustbox}{max width=\textwidth, max height=\textheight}
\rotatebox{270}{
\begin{tabular}{|c|c|c|}\hline
Centralizer $C_{{\bf{G}^*}^{F^*}}(s)$ & $[{\bf{G}^*}^{F^*}:C_{{\bf{G}^*}^{F^*}}(s)]_{p'}$ & Degree of non-real unipotent in $C_{{\bf{G}^*}^{F^*}}(s)$  \\ \hline

$E_8(q)$ & 1 & $\frac{1}{2}q^{11}\Phi_1^7\Phi_3^4\Phi_4^4\Phi_5^2\Phi_7\Phi_8^2\Phi_9\Phi_{12}^2\Phi_{15}\Phi_{20}\Phi_{24}$  \\ \hline

$E_8(q)$ & 1 & $\frac{1}{2}q^{26}\Phi_1^7\Phi_3^4\Phi_4^4\Phi_5^2\Phi_7\Phi_8^2\Phi_9\Phi_{12}^2\Phi_{15}\Phi_{20}\Phi_{24}$ \\ \hline

$E_8(q)$ & 1 & $\frac{1}{3}q^7\Phi_1^6\Phi_2^6\Phi_4^4\Phi_5^2\Phi_7\Phi_8^2\Phi_{10}^2\Phi_{12}\Phi_{14}\Phi_{20}\Phi_{24}$  \\ \hline

$E_8(q)$ & 1 & $\frac{1}{3}q^8\Phi_1^6\Phi_2^6\Phi_4^4\Phi_5^2\Phi_7\Phi_8^2\Phi_{10}^2\Phi_{14}\Phi_{15}\Phi_{20}\Phi_{30}$ \\ \hline

$E_8(q)$ & 1 & $\frac{1}{3}q^{32}\Phi_1^6\Phi_2^6\Phi_4^4\Phi_5^2\Phi_7\Phi_8^2\Phi_{10}^2\Phi_{14}\Phi_{15}\Phi_{20}\Phi_{30}$ \\ \hline

$E_8(q)$ & 1 & $\frac{1}{3}q^{37}\Phi_1^6\Phi_2^6\Phi_4^4\Phi_5^2\Phi_7\Phi_8^2\Phi_{10}^2\Phi_{12}\Phi_{14}\Phi_{20}\Phi_{24}$  \\ \hline

$E_8(q)$ & 1 & $\frac{1}{6}q^{16}\Phi_1^6\Phi_2^8\Phi_4^4\Phi_5^2\Phi_7\Phi_8^2\Phi_{10}^2\Phi_{14}\Phi_{15}\Phi_{18}\Phi_{20}\Phi_{24}$ \\ \hline

$E_8(q)$ & 1 & $\frac{1}{6}q^{16}\Phi_1^6\Phi_2^8\Phi_4^4\Phi_5^2\Phi_6^2\Phi_7\Phi_8^2\Phi_{10}^2\Phi_{12}\Phi_{14}\Phi_{18}\Phi_{20}\Phi_{30}$ \\ \hline

$E_8(q)$ & 1 & $\frac{1}{5}q^{16}\Phi_1^8\Phi_2^8\Phi_3^4\Phi_4^4\Phi_6^4\Phi_7\Phi_8^2\Phi_{9}\Phi_{12}^2\Phi_{14}\Phi_{18}\Phi_{24}$ \\ \hline

$E_8(q)$ & 1 & $\frac{1}{6}q^{16}\Phi_1^8\Phi_2^6\Phi_3^2\Phi_4^4\Phi_5^2\Phi_7\Phi_8^2\Phi_9\Phi_{10}^2\Phi_{12}\Phi_{14}\Phi_{15}\Phi_{20}$ \\ \hline

$E_8(q)$ & 1 & $\frac{1}{6}q^{16}\Phi_1^8\Phi_2^6\Phi_4^4\Phi_5^2\Phi_7\Phi_8^2\Phi_9\Phi_{10}^2\Phi_{14}\Phi_{20}\Phi_{24}\Phi_{30}$ \\ \hline

$E_8(q)$ & 1 & $\frac{1}{4}q^{16}\Phi_1^8\Phi_2^8\Phi_3^4\Phi_5^2\Phi_6^4\Phi_7\Phi_9\Phi_{10}^2\Phi_{14}\Phi_{15}\Phi_{18}\Phi_{30}$ \\ \hline

$E_7(q).A_1(q)$ & $\Phi_3\Phi_4^2\Phi_5\Phi_6\Phi_8\Phi_{10}\Phi_{12}\Phi_{15}\Phi_{20}\Phi_{24}\Phi_{30}$ & $\frac{1}{3}q^7\Phi_1^6\Phi_2^7\Phi_4^2\Phi_5\Phi_6\Phi_7\Phi_8\Phi_{10}\Phi_{14}\Phi_{18}$ \\ \hline

$E_7(q).A_1(q)$ & $\Phi_3\Phi_4^2\Phi_5\Phi_6\Phi_8\Phi_{10}\Phi_{12}\Phi_{15}\Phi_{20}\Phi_{24}\Phi_{30}$ & $\frac{1}{2}q^{11}\Phi_5\Phi_9\Phi_{12}\Phi_8\Phi_7\Phi_4^2\Phi_3^3\Phi_1^7$ \\ \hline

$E_7(q).A_1(q)$ & $\Phi_3\Phi_4^2\Phi_5\Phi_6\Phi_8\Phi_{10}\Phi_{12}\Phi_{15}\Phi_{20}\Phi_{24}\Phi_{30}$ & $\frac{1}{3}q^{16}\Phi_8\Phi_5\Phi_{10}\Phi_7\Phi_{14}\Phi_4^2\Phi_1^6\Phi_2^6$ \\ \hline

$E_7(q).A_1(q)$ & $\Phi_3\Phi_4^2\Phi_5\Phi_6\Phi_8\Phi_{10}\Phi_{12}\Phi_{15}\Phi_{20}\Phi_{24}\Phi_{30}$ & $\frac{1}{3}q^8\Phi_1^6\Phi_2^7\Phi_4^2\Phi_5\Phi_6\Phi_7\Phi_8\Phi_{10}\Phi_{14}\Phi_{18}$ \\ \hline

$E_7(q).A_1(q)$ & $\Phi_3\Phi_4^2\Phi_5\Phi_6\Phi_8\Phi_{10}\Phi_{12}\Phi_{15}\Phi_{20}\Phi_{24}\Phi_{30}$ & $\frac{1}{2}q^{12}\Phi_5\Phi_9\Phi_{12}\Phi_8\Phi_7\Phi_4^2\Phi_3^3\Phi_1^7$\\ \hline

$E_7(q).A_1(q)$ & $\Phi_3\Phi_4^2\Phi_5\Phi_6\Phi_8\Phi_{10}\Phi_{12}\Phi_{15}\Phi_{20}\Phi_{24}\Phi_{30}$ & $\frac{1}{3}q^{17}\Phi_8\Phi_5\Phi_{10}\Phi_7\Phi_{14}\Phi_4^2\Phi_1^6\Phi_2^6$ \\ \hline

$E_7(q).(q-1)$ & $\Phi_2\Phi_3\Phi_4^2\Phi_5\Phi_6\Phi_8\Phi_{10}\Phi_{12}\Phi_{15}\Phi_{20}\Phi_{24}\Phi_{30}$ & 
$\frac{1}{3}q^7\Phi_1^6\Phi_2^7\Phi_4^2\Phi_5\Phi_6\Phi_7\Phi_8\Phi_{10}\Phi_{14}\Phi_{18}$ \\ \hline

$E_7(q).(q-1)$ & $\Phi_2\Phi_3\Phi_4^2\Phi_5\Phi_6\Phi_8\Phi_{10}\Phi_{12}\Phi_{15}\Phi_{20}\Phi_{24}\Phi_{30}$ & 
$\frac{1}{2}q^{11}\Phi_5\Phi_9\Phi_{12}\Phi_8\Phi_7\Phi_4^2\Phi_3^3\Phi_1^7$ \\ \hline

$E_7(q).(q-1)$ & $\Phi_2\Phi_3\Phi_4^2\Phi_5\Phi_6\Phi_8\Phi_{10}\Phi_{12}\Phi_{15}\Phi_{20}\Phi_{24}\Phi_{30}$ & 
$\frac{1}{3}q^{16}\Phi_8\Phi_5\Phi_{10}\Phi_7\Phi_{14}\Phi_4^2\Phi_1^6\Phi_2^6$ \\ \hline

$E_7(q).(q+1)$ & $\Phi_1\Phi_3\Phi_4^2\Phi_5\Phi_6\Phi_8\Phi_{10}\Phi_{12}\Phi_{15}\Phi_{20}\Phi_{24}\Phi_{30}$ & 
$\frac{1}{3}q^7\Phi_1^6\Phi_2^7\Phi_4^2\Phi_5\Phi_6\Phi_7\Phi_8\Phi_{10}\Phi_{14}\Phi_{18}$ \\ \hline

$E_7(q).(q+1)$ & $\Phi_1\Phi_3\Phi_4^2\Phi_5\Phi_6\Phi_8\Phi_{10}\Phi_{12}\Phi_{15}\Phi_{20}\Phi_{24}\Phi_{30}$ & 
$\frac{1}{2}q^{11}\Phi_5\Phi_9\Phi_{12}\Phi_8\Phi_7\Phi_4^2\Phi_3^3\Phi_1^7$ \\ \hline

$E_7(q).(q+1)$ & $\Phi_1\Phi_3\Phi_4^2\Phi_5\Phi_6\Phi_8\Phi_{10}\Phi_{12}\Phi_{15}\Phi_{20}\Phi_{24}\Phi_{30}$ & 
$\frac{1}{3}q^{16}\Phi_8\Phi_5\Phi_{10}\Phi_7\Phi_{14}\Phi_4^2\Phi_1^6\Phi_2^6$ \\ \hline

$E_6(q).A_2(q)$ & $\Phi_2^3\Phi_4^2\Phi_5\Phi_6^2\Phi_7\Phi_8\Phi_{10}^2\Phi_{12}\Phi_{14}\Phi_{15}\Phi_{18}\Phi_{20}\Phi_{24}\Phi_{30}$ & $\frac{1}{3}q^7\Phi_1^6\Phi_2^4\Phi_4^2\Phi_5\Phi_8$ \\ \hline

$E_6(q).A_2(q)$ & $\Phi_2^3\Phi_4^2\Phi_5\Phi_6^2\Phi_7\Phi_8\Phi_{10}^2\Phi_{12}\Phi_{14}\Phi_{15}\Phi_{18}\Phi_{20}\Phi_{24}\Phi_{30}$ & $\frac{1}{3}q^8\Phi_1^6\Phi_2^5\Phi_4^2\Phi_5\Phi_8$ \\ \hline

$E_6(q).A_2(q)$ & $\Phi_2^3\Phi_4^2\Phi_5\Phi_6^2\Phi_7\Phi_8\Phi_{10}^2\Phi_{12}\Phi_{14}\Phi_{15}\Phi_{18}\Phi_{20}\Phi_{24}\Phi_{30}$ & $\frac{1}{3}q^{10}\Phi_1^6\Phi_2^4\Phi_4^2\Phi_5\Phi_8$ \\ \hline

$^2E_6(q). \, ^2A_2(q)$ & $\Phi_1^3\Phi_3^2\Phi_4^2\Phi_5^2\Phi_7\Phi_8\Phi_9\Phi_{10}\Phi_{12}\Phi_{14}\Phi_{15}\Phi_{20}\Phi_{24}\Phi_{30}$ & $\frac{1}{3}q^7\Phi_1^4\Phi_2^6\Phi_4^2\Phi_8\Phi_{10}$ \\ \hline

$^2E_6(q). \, ^2A_2(q)$ & $\Phi_1^3\Phi_3^2\Phi_4^2\Phi_5^2\Phi_7\Phi_8\Phi_9\Phi_{10}\Phi_{12}\Phi_{14}\Phi_{15}\Phi_{20}\Phi_{24}\Phi_{30}$ & $\frac{1}{3}q^8\Phi_1^5\Phi_2^6\Phi_4^2\Phi_8\Phi_{10}$ \\ \hline

$^2E_6(q). \, ^2A_2(q)$ & $\Phi_1^3\Phi_3^2\Phi_4^2\Phi_5^2\Phi_7\Phi_8\Phi_9\Phi_{10}\Phi_{12}\Phi_{14}\Phi_{15}\Phi_{20}\Phi_{24}\Phi_{30}$ & $\frac{1}{3}q^{10}\Phi_1^4\Phi_2^6\Phi_4^2\Phi_8\Phi_{10}$ \\ \hline

$E_6(q).A_1(q).(q-1)$ & $\Phi_2^3\Phi_3\Phi_4^2\Phi_5\Phi_6^2\Phi_7\Phi_8\Phi_{10}^2\Phi_{12}\Phi_{14}\Phi_{15}\Phi_{18}\Phi_{20}\Phi_{24}\Phi_{30}$ & 
$\frac{1}{3}q^7\Phi_1^6\Phi_2^4\Phi_4^2\Phi_5\Phi_8$ \\ \hline

$E_6(q).A_1(q).(q-1)$ & $3\Phi_2^3\Phi_3\Phi_4^2\Phi_5\Phi_6^2\Phi_7\Phi_8\Phi_{10}^2\Phi_{12}\Phi_{14}\Phi_{15}\Phi_{18}\Phi_{20}\Phi_{24}\Phi_{30}$ & 
$\frac{1}{3}q^8\Phi_1^6\Phi_2^4\Phi_4^2\Phi_5\Phi_8$ \\ \hline

$^2E_6(q).A_1(q).(q+1)$ & $\Phi_1^3\Phi_3^2\Phi_4^2\Phi_5^2\Phi_6\Phi_7\Phi_8\Phi_9\Phi_{10}\Phi_{12}\Phi_{14}\Phi_{15}\Phi_{20}\Phi_{24}\Phi_{30}$ & 
$\frac{1}{3}q^7\Phi_1^4\Phi_2^6\Phi_4^2\Phi_8\Phi_{10}$ \\ \hline

$^2E_6(q).A_1(q).(q+1)$ & $\Phi_1^3\Phi_3^2\Phi_4^2\Phi_5^2\Phi_6\Phi_7\Phi_8\Phi_9\Phi_{10}\Phi_{12}\Phi_{14}\Phi_{15}\Phi_{20}\Phi_{24}\Phi_{30}$ & 
$\frac{1}{3}q^8\Phi_1^4\Phi_2^6\Phi_4^2\Phi_8\Phi_{10}$ \\ \hline

$E_6(q) .(q-1)^2$ & $\Phi_2^4\Phi_3\Phi_4^2\Phi_5\Phi_6^2\Phi_7\Phi_8\Phi_{10}^2\Phi_{12}\Phi_{14}\Phi_{15}\Phi_{18}\Phi_{20}\Phi_{24}\Phi_{30}$ & $\frac{1}{3}q^7\Phi_1^6\Phi_2^4\Phi_4^2\Phi_5\Phi_8$ \\ \hline

$E_6(q).(q^2-1)$ & $\Phi_1\Phi_2^3\Phi_3\Phi_4^2\Phi_5\Phi_6^2\Phi_7\Phi_8\Phi_{10}^2\Phi_{12}\Phi_{14}\Phi_{15}\Phi_{18}\Phi_{20}\Phi_{24}\Phi_{30}$ & $\frac{1}{3}q^7\Phi_1^6\Phi_2^4\Phi_4^2\Phi_5\Phi_8$ \\ \hline

$^2E_6(q) .(q^2-1)$ & $\Phi_1^3\Phi_2\Phi_3^2\Phi_4^2\Phi_5^2\Phi_6\Phi_7\Phi_8\Phi_9\Phi_{10}\Phi_{12}\Phi_{14}\Phi_{15}\Phi_{20}\Phi_{24}\Phi_{30}$ & $\frac{1}{3}q^7\Phi_1^4\Phi_2^6\Phi_4^2\Phi_8\Phi_{10}$ \\ \hline

$^2E_6(q).(q^2-q+1)$ & $\Phi_1^4\Phi_2^2\Phi_3^2\Phi_4^2\Phi_5^2\Phi_7\Phi_8\Phi_9\Phi_{10}\Phi_{12}\Phi_{14}\Phi_{15}\Phi_{20}\Phi_{24}\Phi_{30}$ & $\frac{1}{3}q^7\Phi_1^4\Phi_2^6\Phi_4^2\Phi_8\Phi_{10}$ \\ \hline

$E_6(q).(q^2+q+1)$ & $\Phi_1^2\Phi_2^4\Phi_4^2\Phi_5\Phi_6^2\Phi_7\Phi_8\Phi_{10}^2\Phi_{12}\Phi_{14}\Phi_{15}\Phi_{18}\Phi_{20}\Phi_{24}\Phi_{30}$ & $\frac{1}{3}q^7\Phi_1^6\Phi_2^4\Phi_4^2\Phi_5\Phi_8$ \\ \hline

$^2E_6(q) .(q+1)^2$ & $\Phi_1^4\Phi_3^2\Phi_4^2\Phi_5^2\Phi_6\Phi_7\Phi_8\Phi_9\Phi_{10}\Phi_{12}\Phi_{14}\Phi_{15}\Phi_{20}\Phi_{24}\Phi_{30}$ & $\frac{1}{3}q^7\Phi_1^4\Phi_2^6\Phi_4^2\Phi_8\Phi_{10}$ \\ \hline
\end{tabular}
}
\end{adjustbox}
\end{table}
\end{center}
\newpage
\begin{center}
\begin{table}[h!]
\begin{adjustbox}{max width=\textwidth, max height=\textheight}
\rotatebox{270}{
\begin{tabular}{|c|c|c|c|c|c|}\hline
Degree of non-real in $E_8(q)$ & $q\equiv 1 \bmod 6$ & $q\equiv 2 \bmod 6$ & $q\equiv 3 \bmod 6$ & $q\equiv 4 \bmod 6$ & $q\equiv 5 \bmod 6$ \\ \hline

$\frac{1}{2}q^{11}\Phi_1^7\Phi_3^4\Phi_4^4\Phi_5^2\Phi_7\Phi_8^2\Phi_9\Phi_{12}^2\Phi_{15}\Phi_{20}\Phi_{24}$ & 2 & 2 &2 & 2 & 2 \\ \hline

 $\frac{1}{2}q^{26}\Phi_1^7\Phi_3^4\Phi_4^4\Phi_5^2\Phi_7\Phi_8^2\Phi_9\Phi_{12}^2\Phi_{15}\Phi_{20}\Phi_{24}$ & 2 & 2 &2 & 2 & 2 \\ \hline

$\frac{1}{3}q^7\Phi_1^6\Phi_2^6\Phi_4^4\Phi_5^2\Phi_7\Phi_8^2\Phi_{10}^2\Phi_{12}\Phi_{14}\Phi_{20}\Phi_{24}$ & 2 & 2 &2 & 2 & 2 \\ \hline

$\frac{1}{3}q^8\Phi_1^6\Phi_2^6\Phi_4^4\Phi_5^2\Phi_7\Phi_8^2\Phi_{10}^2\Phi_{14}\Phi_{15}\Phi_{20}\Phi_{30}$ & 2 & 2 &2 & 2 & 2 \\ \hline

$\frac{1}{3}q^{32}\Phi_1^6\Phi_2^6\Phi_4^4\Phi_5^2\Phi_7\Phi_8^2\Phi_{10}^2\Phi_{14}\Phi_{15}\Phi_{20}\Phi_{30}$ & 2 & 2 &2 & 2 & 2 \\ \hline

$\frac{1}{3}q^{37}\Phi_1^6\Phi_2^6\Phi_4^4\Phi_5^2\Phi_7\Phi_8^2\Phi_{10}^2\Phi_{12}\Phi_{14}\Phi_{20}\Phi_{24}$ & 2 & 2 &2 & 2 & 2 \\ \hline

$\frac{1}{6}q^{16}\Phi_1^6\Phi_2^8\Phi_4^4\Phi_5^2\Phi_7\Phi_8^2\Phi_{10}^2\Phi_{14}\Phi_{15}\Phi_{18}\Phi_{20}\Phi_{24}$ & 2 & 2 &2 & 2 & 2 \\ \hline

 $\frac{1}{6}q^{16}\Phi_1^6\Phi_2^8\Phi_4^4\Phi_5^2\Phi_6^2\Phi_7\Phi_8^2\Phi_{10}^2\Phi_{12}\Phi_{14}\Phi_{18}\Phi_{20}\Phi_{30}$ & 2 & 2 &2 & 2 & 2 \\ \hline

$\frac{1}{5}q^{16}\Phi_1^8\Phi_2^8\Phi_3^4\Phi_4^4\Phi_6^4\Phi_7\Phi_8^2\Phi_{9}\Phi_{12}^2\Phi_{14}\Phi_{18}\Phi_{24}$ & 4 & 4 & 4 & 4 & 4 \\ \hline

$\frac{1}{6}q^{16}\Phi_1^8\Phi_2^6\Phi_3^2\Phi_4^4\Phi_5^2\Phi_7\Phi_8^2\Phi_9\Phi_{10}^2\Phi_{12}\Phi_{14}\Phi_{15}\Phi_{20}$ & 2 & 2 &2 & 2 & 2 \\ \hline

$\frac{1}{6}q^{16}\Phi_1^8\Phi_2^6\Phi_4^4\Phi_5^2\Phi_7\Phi_8^2\Phi_9\Phi_{10}^2\Phi_{14}\Phi_{20}\Phi_{24}\Phi_{30}$ & 2 & 2 &2 & 2 & 2 \\ \hline

$\frac{1}{4}q^{16}\Phi_1^8\Phi_2^8\Phi_3^4\Phi_5^2\Phi_6^4\Phi_7\Phi_9\Phi_{10}^2\Phi_{14}\Phi_{15}\Phi_{18}\Phi_{30}$ & 2 & 2 &2 & 2 & 2 \\ \hline

$\frac{1}{3}q^7\Phi_1^6\Phi_2^7\Phi_3\Phi_4^4\Phi_5^2\Phi_6^2\Phi_7\Phi_8^2\Phi_{10}^2\Phi_{12}\Phi_{14}\Phi_{15}\Phi_{18}\Phi_{20}\Phi_{24}\Phi_{30}$ & 2 & 0 & 2 & 0 & 2 \\ \hline

$\frac{1}{2}q^{11}\Phi_1^7\Phi_3^4\Phi_4^4\Phi_5^2\Phi_6\Phi_7\Phi_8^2\Phi_9\Phi_{10}\Phi_{12}\Phi_{15}\Phi_{20}\Phi_{24}\Phi_{30}$ & 2 & 0 & 2 & 0 & 2 \\ \hline

$\frac{1}{3}q^{16}\Phi_1^6\Phi_2^6\Phi_3\Phi_4^4\Phi_5^2\Phi_6\Phi_7\Phi_8^2\Phi_{10}^2\Phi_{12}\Phi_{14}\Phi_{15}\Phi_{20}\Phi_{24}\Phi_{30}$ & 2 & 0 & 2 & 0 & 2 \\ \hline

$\frac{1}{3}q^8\Phi_1^6\Phi_2^7\Phi_3\Phi_4^4\Phi_5^2\Phi_6^2\Phi_7\Phi_8^2\Phi_{10}^2\Phi_{12}\Phi_{14}\Phi_{15}\Phi_{18}\Phi_{20}\Phi_{24}\Phi_{30}$ & 2 & 0 & 2 & 0 & 2 \\ \hline

$\frac{1}{2}q^{12}\Phi_1^7\Phi_3^4\Phi_4^4\Phi_5^2\Phi_6\Phi_7\Phi_8^2\Phi_9\Phi_{10}\Phi_{12}\Phi_{15}\Phi_{20}\Phi_{24}\Phi_{30}$ & 2 & 0 & 2 & 0 & 2 \\ \hline

$\frac{1}{3}q^{17}\Phi_1^6\Phi_2^6\Phi_3\Phi_4^4\Phi_5^2\Phi_6\Phi_7\Phi_8^2\Phi_{10}^2\Phi_{12}\Phi_{14}\Phi_{15}\Phi_{20}\Phi_{24}\Phi_{30}$ & 2 & 0 & 2 & 0 & 2 \\ \hline

$\frac{1}{3}q^7\Phi_1^6\Phi_2^8\Phi_3\Phi_4^4\Phi_5^2\Phi_6^2\Phi_7\Phi_8^2\Phi_{10}^2\Phi_{12}\Phi_{14}\Phi_{15}\Phi_{18}\Phi_{20}\Phi_{24}\Phi_{30}$ & $q-3$ & $q-2$ & $q-3$ & $q-2$ & $q-3$ \\ \hline

$\frac{1}{2}q^{11}\Phi_1^7\Phi_2\Phi_3^4\Phi_4^4\Phi_5^2\Phi_6\Phi_7\Phi_8^2\Phi_9\Phi_{10}\Phi_{12}\Phi_{15}\Phi_{20}\Phi_{24}\Phi_{30}$ &
$q-3$ & $q-2$ & $q-3$ & $q-2$ & $q-3$  \\ \hline

$\frac{1}{3}q^{16}\Phi_1^6\Phi_2^7\Phi_3\Phi_4^4\Phi_5^2\Phi_6\Phi_7\Phi_8^2\Phi_{10}^2\Phi_{12}\Phi_{14}\Phi_{15}\Phi_{20}\Phi_{24}\Phi_{30}$ &
$q-3$ & $q-2$ & $q-3$ & $q-2$ & $q-3$  \\ \hline

$\frac{1}{3}q^7\Phi_1^7\Phi_2^7\Phi_3\Phi_4^4\Phi_5^2\Phi_6^2\Phi_7\Phi_8^2\Phi_{10}^2\Phi_{12}\Phi_{14}\Phi_{15}\Phi_{18}\Phi_{20}\Phi_{24}\Phi_{30}$ &
$q-1$ & $q$ & $q-1$ & $q$ & $q-1$ \\ \hline

$\frac{1}{2}q^{11}\Phi_1^8\Phi_3^4\Phi_4^4\Phi_5^2\Phi_6\Phi_7\Phi_8^2\Phi_9\Phi_{10}\Phi_{12}^2\Phi_{15}\Phi_{20}\Phi_{24}\Phi_{30}$ &
$q-1$ & $q$ & $q-1$ & $q$ & $q-1$ \\ \hline

$\frac{1}{3}q^{16}\Phi_1^7\Phi_2^6\Phi_3\Phi_4^4\Phi_5^2\Phi_6\Phi_7\Phi_8^2\Phi_{10}^2\Phi_{12}\Phi_{14}\Phi_{15}\Phi_{20}\Phi_{24}\Phi_{30}$ &
$q-1$ & $q$ & $q-1$ & $q$ & $q-1$ \\ \hline

$\frac{1}{3}q^7\Phi_1^6\Phi_2^7\Phi_4^4\Phi_5^2\Phi_6^2\Phi_7\Phi_8^2 \Phi_{10}^2\Phi_{12}\Phi_{14}\Phi_{15}\Phi_{18}\Phi_{20}\Phi_{24}\Phi_{30}$ & 2 & 0 & 0 & 2 & 0 \\ \hline

$\frac{1}{3}q^8\Phi_1^6\Phi_2^8\Phi_4^4\Phi_5^2\Phi_6^2\Phi_7\Phi_8^2 \Phi_{10}^2\Phi_{12}\Phi_{14}\Phi_{15}\Phi_{18}\Phi_{20}\Phi_{24}\Phi_{30}$ & 2 & 0 & 0 & 2 & 0 \\ \hline

$\frac{1}{3}q^{10}\Phi_1^6\Phi_2^7\Phi_4^4\Phi_5^2\Phi_6^2\Phi_7\Phi_8^2 \Phi_{10}^2\Phi_{12}\Phi_{14}\Phi_{15}\Phi_{18}\Phi_{20}\Phi_{24}\Phi_{30}$ & 2 & 0 & 0 & 2 & 0 \\ \hline

$\frac{1}{3}q^7\Phi_1^7\Phi_2^6\Phi_3^2\Phi_4^4\Phi_5^2\Phi_7\Phi_8^2\Phi_9\Phi_{10}^2\Phi_{12}\Phi_{14}\Phi_{15}\Phi_{20}\Phi_{24}\Phi_{30}$ & 0 & 2 & 0 & 0 & 2 \\ \hline

$\frac{1}{3}q^8\Phi_1^8\Phi_2^6\Phi_3^2\Phi_4^4\Phi_5^2\Phi_7\Phi_8^2\Phi_9\Phi_{10}^2\Phi_{12}\Phi_{14}\Phi_{15}\Phi_{20}\Phi_{24}\Phi_{30}$ & 0 & 2 & 0 & 0 & 2  \\ \hline

$\frac{1}{3}q^{10}\Phi_1^7\Phi_2^6\Phi_3^2\Phi_4^4\Phi_5^2\Phi_7\Phi_8^2\Phi_9\Phi_{10}^2\Phi_{12}\Phi_{14}\Phi_{15}\Phi_{20}\Phi_{24}\Phi_{30}$ & 0 & 2 & 0 & 0 & 2  \\ \hline

$\frac{1}{3}q^7\Phi_1^6\Phi_2^7\Phi_3\Phi_4^4\Phi_5^2\Phi_6^2\Phi_7\Phi_8^2\Phi_{10}^2\Phi_{12}\Phi_{14}\Phi_{15}\Phi_{18}\Phi_{20}\Phi_{24}\Phi_{30}$ & $q-5$ & $q-2$ & $q-3$ & $q-4$ & $q-3$ \\ \hline

$\frac{1}{3}q^8\Phi_1^6\Phi_2^7\Phi_3\Phi_4^4\Phi_5^2\Phi_6^2\Phi_7\Phi_8^2\Phi_{10}^2\Phi_{12}\Phi_{14}\Phi_{15}\Phi_{18}\Phi_{20}\Phi_{24}\Phi_{30}$ & $q-5$ & $q-2$ & $q-3$ & $q-4$ & $q-3$\\ \hline

$\frac{1}{3}q^7\Phi_1^7\Phi_2^6\Phi_3^2\Phi_4^4\Phi_5^2\Phi_6\Phi_7\Phi_8^2\Phi_9\Phi_{10}^2\Phi_{12}\Phi_{14}\Phi_{15}\Phi_{20}\Phi_{24}\Phi_{30}$ & $q-1$ & $q-2$ & $q-1$ & $q$ & $q-3$\\ \hline

$\frac{1}{3}q^8\Phi_1^7\Phi_2^6\Phi_3^2\Phi_4^4\Phi_5^2\Phi_6\Phi_7\Phi_8^2\Phi_9\Phi_{10}^2\Phi_{12}\Phi_{14}\Phi_{15}\Phi_{20}\Phi_{24}\Phi_{30}$ & $q-1$ & $q-2$ & $q-1$ & $q$ & $q-3$ \\ \hline

$\frac{1}{3}q^7\Phi_1^6\Phi_2^8\Phi_3\Phi_4^4\Phi_5^2\Phi_6^2\Phi_7\Phi_8^2\Phi_{10}^2\Phi_{12}\Phi_{14}\Phi_{15}\Phi_{18}\Phi_{20}\Phi_{24}\Phi_{30}$ & $\frac{1}{6}(q^2-8q+19)$ & $\frac{1}{6}(q-2)(q-6)$ & $\frac{1}{6}(q-3)(q-5)$ & $\frac{1}{6}(q-4)^2$ & $\frac{1}{6}(q-3)(q-5)$ \\ \hline

$\frac{1}{3}q^7\Phi_1^7\Phi_2^7\Phi_3\Phi_4^4\Phi_5^2\Phi_6^2\Phi_7\Phi_8^2\Phi_{10}^2\Phi_{12}\Phi_{14}\Phi_{15}\Phi_{18}\Phi_{20}\Phi_{24}\Phi_{30}$ & $\frac{1}{2}(q-1)^2$ & $\frac{1}{2}q(q-2)$ & $\frac{1}{2}(q-1)^2$ & $\frac{1}{2}q(q-2)$ & $\frac{1}{2}(q-1)^2$ \\ \hline

$\frac{1}{3}q^7\Phi_1^7\Phi_2^7\Phi_3^2\Phi_4^4\Phi_5^2\Phi_6\Phi_7\Phi_8^2\Phi_9\Phi_{10}^2\Phi_{12}\Phi_{14}\Phi_{15}\Phi_{20}\Phi_{24}\Phi_{30}$ & $\frac{1}{2}(q-1)^2$ & $\frac{1}{2}q(q-2)$ & $\frac{1}{2}(q-1)^2$ & $\frac{1}{2}q(q-2)$ & $\frac{1}{2}(q-1)^2$ \\ \hline

$\frac{1}{3}q^7\Phi_1^8\Phi_2^8\Phi_3^2\Phi_4^4\Phi_5^2\Phi_7\Phi_8^2\Phi_9\Phi_{10}^2\Phi_{12}\Phi_{14}\Phi_{15}\Phi_{20}\Phi_{24}\Phi_{30}$ & $\frac{1}{3}q(q-1)$ & $\frac{1}{3}(q+1)(q-2)$ & $\frac{1}{3}q(q-1)$ & $\frac{1}{3}q(q-1)$ & $\frac{1}{3}(q+1)(q-2)$ \\ \hline

$\frac{1}{3}q^7\Phi_1^8\Phi_2^8\Phi_4^4\Phi_5^2\Phi_6^2\Phi_7\Phi_8^2\Phi_{10}^2\Phi_{12}\Phi_{14}\Phi_{15}\Phi_{18}\Phi_{20}\Phi_{24}\Phi_{30}$ & $\frac{1}{3}(q-1)(q+2)$ & $\frac{1}{3}q(q+1)$ & $\frac{1}{3}q(q+1)$ & $\frac{1}{3}(q-1)(q+2)$ & $\frac{1}{3}q(q+1)$\\ \hline

$\frac{1}{3}q^7\Phi_1^8\Phi_2^6\Phi_3^2\Phi_4^4\Phi_5^2\Phi_6\Phi_7\Phi_8^2\Phi_9\Phi_{10}^2\Phi_{12}\Phi_{14}\Phi_{15}\Phi_{20}\Phi_{24}\Phi_{30}$ & $\frac{1}{6}(q-1)(q-3)$ & $\frac{1}{6}(q-2)^2$ & $\frac{1}{6}(q-1)(q-3)$ & $\frac{1}{6}q(q-4)$ & $\frac{1}{6}(q^2-4q+7)$ \\ \hline
\end{tabular}
}
\end{adjustbox}
\end{table}
\end{center}

\bigskip

\noindent
\begin{tabular}{ll}
\textsc{Department of Mathematics}\\ 
\textsc{College of William and Mary}\\
\textsc{P. O. Box 8795} \\
\textsc{Williamsburg, VA  23187}\\
{\em e-mail}: {\tt sjtrefethen@wm.edu},   {\tt vinroot@math.wm.edu}\\
\end{tabular}
\end{document}